\documentclass[12pt,twoside]{article}
\usepackage{etex}
\usepackage[dvipsnames]{xcolor}
\usepackage{color}
\usepackage{booktabs}
\usepackage{amsthm}
\usepackage{amsfonts,amssymb,amsmath,amsxtra,url,float}
\usepackage{indentfirst}
\allowdisplaybreaks[4]
\usepackage[colorlinks, linkcolor=magenta, anchorcolor=Periwinkle, citecolor=violet, urlcolor=blue]{hyperref}
\usepackage{enumitem}
\usepackage{geometry} 
\usepackage{rotating}
\usepackage{lscape}
\usepackage{multirow}
\usepackage{graphicx}
\usepackage{subfigure}
\usepackage{adjustbox}
\usepackage{mathrsfs}
\usepackage{fancyhdr}
\usepackage[all]{xy}
\usepackage{dsfont}
\usepackage{marginnote}
\usepackage{changepage}

\usepackage{enumitem}
\setenumerate[1]{itemsep=0pt,partopsep=0pt,parsep=\parskip,topsep=3pt}
\setitemize[1]{itemsep=0pt,partopsep=0pt,parsep=\parskip,topsep=3pt}
\setdescription{itemsep=0pt,partopsep=0pt,parsep=\parskip,topsep=3pt}
\setlist[itemize]{leftmargin=35pt}
\setlist[enumerate]{leftmargin=35pt}

\usepackage{imakeidx}
\makeindex[columns=2, intoc]
\usepackage{tikz}
\usepackage{pgfplots}
\def\Pic{Figure~}

\theoremstyle{plain}
\newtheorem{theorem}{Theorem}[section]
\newtheorem{lemma}[theorem]{Lemma}
\newtheorem{notation}[theorem]{Notation}
\newtheorem{corollary}[theorem]{Corollary}
\newtheorem{proposition}[theorem]{Proposition}
\newtheorem{definition}[theorem]{Definition}
\newtheorem{remark}[theorem]{Remark}
\newtheorem{example}[theorem]{Example}

\numberwithin{equation}{section}
\numberwithin{figure}{section}

\DeclareMathOperator{\image}{\mathrm{Im}}
\DeclareMathOperator{\kernel}{\mathrm{Ker}}
\DeclareMathOperator{\Hom}{\mathrm{Hom}}

\DeclareMathOperator{\End}{End}

\def\top{\operatorname{top}}
\def\rad{\operatorname{rad}}

\def\ind{\mathsf{ind}}

\newcommand{\C}{\mathscr{C}}          
\newcommand{\calR}{\mathcal{R}}       
\newcommand{\rmI}{\mathrm{I}}         
\def\bC{\bar{C}}
\def\tC{\widetilde{C}}
\newcommand{\Q}{\mathcal{Q}}
\newcommand{\I}{\mathcal{I}}
\def\s{\mathfrak{s}}
\def\t{\mathfrak{t}}
\def\e{\epsilon}
\def\nota{\star}

\def\op{\mathrm{op}}
\def\CMA{\mathrm{CMA}}
\def\itLamb{\mathit{\Lambda}}

\newcommand{\modcat}{\mathsf{mod}}

\newcommand{\Dcat}{\mathsf{D}}
\newcommand{\comp}[1]{\pmb{#1}^{\bullet}}
\newcommand{\coH}{\mathrm{H}}
\newcommand{\hw}{\mathrm{hw}}
\newcommand{\pdim}{\mathrm{proj.dim}}
\newcommand{\idim}{\mathrm{inj.dim}}
\newcommand{\gldim}{\mathrm{gl.dim}}
\newcommand{\glhw}{\mathrm{gl.hw}}
\newcommand{\findim}{\mathrm{fin.dim}}

\newcommand{\KGdim}{\mathrm{KGdim}~}
\newcommand{\Gproj}{\mathsf{G}\text{-}\mathsf{proj}}
\newcommand{\proj}{\mathsf{proj}}

\newcommand{\str}{\mathbf{s}}

\newcommand{\M}{\mathbb{M}}
\newcommand{\X}{\mathbb{X}}
\newcommand{\Str}{\mathsf{Str}}
\newcommand{\Band}{\mathsf{Ban}}
\newcommand{\hStr}{\mathsf{h.Str}}
\newcommand{\hBand}{\mathsf{h.Ban}}
\newcommand{\emb}{\mathfrak{e}}
\newcommand{\res}{\mathfrak{res}}
\def\sfs{\mathsf{s}}
\def\sfb{\mathsf{b}}
\def\sfw{\mathsf{w}}
\def\sfh{\mathsf{h}}

\newcommand{\defines}[1]{{\it\color{red}#1\index{#1}}}
\newcommand{\sm}[1]{\left(\begin{smallmatrix}#1\end{smallmatrix}\right)_{A}}

\newcommand{\To}[1]{\mathop{\longrightarrow}\limits^{#1}}

\newcommand{\kk}{\mathbf{k}}
\newcommand{\shadow}[1]{{\color{#1}$\blacksquare\!\!\!\blacksquare$}}
\def\NN{\mathbb{N}}
\def\ZZ{\mathbb{Z}}

\def\Ab{\mathsf{Ab}}
\def\>={\geqslant}
\def\=<{\leqslant}

\def\orcid{\begin{tikzpicture}[baseline=-1mm] \filldraw[Green!35] (0,0) circle (5pt); \filldraw[white] (0,0) node{\tiny\textbf{iD}}; \end{tikzpicture}}
\newcommand{\ORCID}[1]{ORCID: \href{https://orcid.org/#1}{#1}}
\newcommand{\ORCIDNOTATION}[1]{\href{https://orcid.org/#1}{\orcid}}

\def\headertitle{Recollements of Cohen--Macaulay Auslander algebras for gentle algebras}
\def\fstpage{1}
\def\page{$\begin{matrix} {\color{white}0} \\ \thepage \end{matrix}$}

\def\EnglishTitle{Recollements of  Cohen--Macaulay Auslander algebras for gentle algebras}
\def\EnglishFundings{Yu-Zhe Liu is supported by
the National Natural Science Foundation of China (Grant Nos. 12401042 and 12561008),
the Guizhou Provincial Basic Research Program (Natural Science) (Grant Nos. ZD[2025]085 and ZK[2024]YiBan066),
the Guizhou Provincial Major Project of Basic Research Program (Grant No. VZD[2026]001);
Xin Ma is supported by
the Central Plains Science and Technology Innovation Youth Top-notch Talent (2024ZYBJRC002),
and the Natural Science Foundation of Henan (262300421832);
and Guiqi Shi is supported by the Scientific Research Foundation of Guizhou University (Grant No. [2023]16).
}

\def\FirstAuthorORICD{0009-0007-7780-0160}

\def\FourthAuthorORICD{0009-0005-1110-386X}

\def\PaperAuthorsENname{
	Jiacheng Xu$^{~\ref{Author1}, \ORCIDNOTATION{\FirstAuthorORICD}\ref{orcid1}}$,
	Yu-Zhe Liu$^{~\ref{Author1}, \ORCIDNOTATION{\FourthAuthorORICD}\ref{orcid1}~\ref{CorrespondingAuthor}}$,
	Xin Ma$^{~\ref{Author2}}$,
	Guiqi Shi$^{~\ref{Author1}}$
}
\def\FirstEnOrgani{School of Mathematics and Statistics, Guizhou University, Guiyang 550025, Guizhou, China}
\def\SecondEnOrgani{College of Science, Henan University of Engineering, Zhengzhou 451191, Henan, China}
\def\FirstEmail{\url{xjcgzu823@163.com} (J.Xu)}
\def\SecondEmail{\url{yzliu3@163.com}/\url{liuyz@gzu.edu.cn} (Y.-Z. Liu)}
\def\ThirdEmail{\url{maxin@haue.edu.cn} (X. Ma)}
\def\FourthEmail{\url{gqshi@gzu.edu.cn} (G. Shi)}

\title{\bf \EnglishTitle$^{\color{red}\dag}$ \footnotetext[2]{\tiny \EnglishFundings}}
\author{\PaperAuthorsENname}
\date{}

\begin{document}
	
\thispagestyle{empty}
\maketitle
	
\begin{enumerate}[label=\textbf{\color{red}$\ddag$}]
  \item \footnotesize \begin{center} Corresponding author \end{center} \label{CorrespondingAuthor}
\end{enumerate}
	
\begin{enumerate}[leftmargin=6.2cm] \footnotesize
  \item[\orcid] \ORCID{\FirstAuthorORICD} \label{orcid1}
  \item[\orcid] \ORCID{\FourthAuthorORICD} \label{orcid1}
\end{enumerate}
	
\vspace{2mm}
	
\begin{enumerate}[label=\textbf{\color{red}\arabic*}] \footnotesize
  \item \begin{center} \FirstEnOrgani \\ E-mail: \FirstEmail; \SecondEmail;  \FourthEmail \end{center} \label{Author1}
  \item \begin{center} \SecondEnOrgani \\ E-mail: \ThirdEmail \end{center} \label{Author2}
\end{enumerate}
	
\vspace{1mm}
	
\begin{adjustwidth}{1cm}{1cm}
\noindent \footnotesize
\textbf{Abstract}:
We construct two recollements of module categories for the Cohen--Macaulay Auslander algebra $A^{\mathrm{CMA}}$ of a gentle algebra $A$.
In this paper, we establish three equivalent characterizations for the quotient algebra $A^{\mathrm{CMA}}/A^{\mathrm{CMA}}(1-\epsilon_{\star}) A^{\mathrm{CMA}}$ of the CM--Auslander algebra of $A$ to be quasi-tilted,
precisely, the following statements are equivalent:
\begin{itemize}
  \item[(1)] $A^{\mathrm{CMA}}/A^{\mathrm{CMA}}(1-\epsilon_{\star}) A^{\mathrm{CMA}}$ is quasi-tilted;
  \item[(2)] $\findim A\leqslant 2$, and for each forbidden $A$-module $M$, $\mathrm{proj.dim}M+\mathrm{inj.dim}M\leqslant 2$;
  \item[(3)] for any homotopy string/band $\mathsf{h}$ none of whose arrows lie on any forbidden cycle,
    the cohomological width of the indecomposable object in $\mathsf{D}^b(A)$ corresponding to $\mathsf{h}$ is $\leqslant 2$.
\end{itemize}
Moreover, we prove that the Krull--Gabriel dimension of $A$ is bounded by 2 if and only if the Krull--Gabriel dimension of $A^{\mathrm{CMA}}$ is bounded by 2 in the case where $A$ is gentle one-cycle.
\vspace{1mm}

\label{2020MSC}
\vspace{1mm}
		
 \label{Keywords}
\end{adjustwidth}
	
\linespread{1.3}

\setcounter{tocdepth}{3}
\setcounter{secnumdepth}{3} 
\tableofcontents

	
\section*{Introduction}

Recollement was first introduced by Be\u{\i}linson, Bernstein, Deligne, and Gabber in \cite{BBDG1982},
encompassing both its formulation in Abelian categories and triangulated categories,
as part of their study of derived categories of sheaves on singular spaces.
In this case, a well-known example of a recollement situation of Abelian categories
appeared in the construction of perverse sheaves given by MacPherson and Vilonen \cite{MV1986}.
This concept serves to connect the homological theories \cite[etc]{CX2017,ZZ2020,ZZ2020JAA},
Gorenstein homological theories \cite[etc]{Zhang2013Gorenstein,PSS2014,Qin2020CMAuslander},
homotopy theories \cite[etc]{IKM2011},
weighted projective lines \cite[etc]{Chen2012,LO2017,Ruan2021},
representation theories \cite[etc]{Jor2006,Psa2014,CX2019,PV2019,CJS2022},
and tilting theories \cite[etc]{CX2012,MZ2020,MXZ2022,CX2023} of three algebras.
In particular, Psaroudakis proved that a recollement of three module categories of algebras has the following form
\begin{align} \label{recollement}
\calR(A,e):= \
\xymatrix@C=2cm{
  \modcat(A/AeA) \ar[r]^{\emb}_{\text{embedding}}
& \modcat(A) \ar[r]^{\res_e:=(-)e}_{\text{restriction}}
  \ar@/_1.5pc/[l]_{-\otimes_A A/AeA}
  \ar@/^1.5pc/[l]^{\Hom_{A}(A/AeA,-)}
& \modcat(eAe),
  \ar@/_1.5pc/[l]_{-\otimes_{eAe} eA}
  \ar@/^1.5pc/[l]^{\Hom_{eAe}(Ae,-)}
}
\end{align}
where $e$ is an idempotent in $A$, and $\emb$ denotes the embedding functor
given by the fact that each $A/AeA$-module is an $A$-module, see \cite[Example 2.7]{Psa2014}.

In our paper, all algebras we consider are finite-dimensional algebras over
an algebraically closed field $\kk$, and all modules over $A$ are finitely generated right $A$-modules.
We will point out that a recollement is a very powerful tool
that can effectively link a gentle algebra and its CM--Auslander algebra.
This work is based on G-projective modules over a gentle algebra in \cite{Kal2015},
the description of the CM--Auslander algebra of a gentle algebra in \cite{CL2017,CL2019},
and the recollement for gentle algebras presented in \cite{LLM2025}.
Let $A$ be a gentle algebra, and let $(\Q,\I)$ be its bound quiver.
In \cite{Kal2015}, Kalck showed that a non-projective indecomposable G-projective module
is isomorphic to $\alpha A$, where $\alpha$ is an arrow on a forbidden cycle
(the definition of a forbidden cycle is given in Definition \ref{def:forb.path} in this paper).
Let $A^{\CMA}$ be the CM--Auslander algebra of $A$ and $(\Q^{\CMA},\I^{\CMA})$ its bound quiver.
In \cite{CL2017,CL2019}, Chen and Lu proved that the vertex set of $(\Q^{\CMA},\I^{\CMA})$
corresponds one-to-one to the union of $\Q_0$ and $\ind(\Gproj(A))\backslash\ind(\proj(A))$.
For a gentle algebra $A$, let $\e_{\nota}$ be the sum of all idempotents in $A^{\CMA}$
corresponding to the vertices in $\Q_0$ (see (\ref{formula:idempotent})),
the first result of this paper is the following.

\begin{theorem} Let $A$ be a gentle algebra. Then the following statements hold.
\begin{itemize}
  \item[\rm(1)]{\rm(Lemma \ref{lemm:recollement CMA})}
    $\calR(A^{\CMA},\e_{\nota})$ is a recollement of three finitely generated module categories
    $\modcat(A^{\CMA}/A^{\CMA}\e_{\nota}A^{\CMA})$, $\modcat(A^{\CMA})$, and $\modcat(A)$.

  \item[\rm(2)]{\rm(Proposition \ref{prop:recoll CMA} (1))}
    For any vertex $v$ corresponding to a G-projective module $G\in\ind(\Gproj(A))\backslash\ind(\proj(A))$,
    the $A$-module obtained by applying the functor $\res_{\e_{\nota}}$ to $e_vA^{\CMA}$ is G-projective.

  \item[\rm(3)]{\rm(Proposition \ref{prop:recoll CMA} (2))}
    For any G-projective module $G\in\ind(\Gproj(A))\backslash\ind(\proj(A))$,
    the $C$-module obtained by applying the functor $-\otimes_A \e_{\nota}C$ to $G$ is not G-projective.
\end{itemize}
\end{theorem}

We call an indecomposable module $M$ over a gentle algebra $A$ is a \defines{forbidden module} if its is a string module
such that there is an arrow on the string corresponding to $M$ such that this arrow is on a forbidden cycle,
see Definition \ref{def:forb.mod.} of this paper.
Another recollement concerning the CM--Auslander algebra $A^{\CMA}$ of a gentle algebra $A$ is $\calR(A^{\CMA},1-\e_{\nota})$, which relates the quasi-tilting property of $A^{\CMA}/A^{\CMA}(1-\e_{\nota})A^{\CMA}$ and the homological properties of the quotient algebra $A$, see the following theorem.

\begin{theorem}[{\rm Theorems \ref{thm:main}, \ref{thm:main2}}]
Let $A$ be a gentle algebra. Then the following statements are equivalent:
\begin{itemize}
  \item[\rm(1)] the quotient $A^{\CMA}/A^{\CMA}(1-\e_{\nota})A^{\CMA}$ is quasi-tilted;
  \item[\rm(2)] the finitistic dimension of $A$ is at most $2$ and every non-forbidden indecomposable $A$-module $M$ with finite projective dimension and injective dimension satisfies $\pdim M + \idim M \=< 3$;
  \item[\rm(3)] for any homotopy string/band $\sfh$
    whose all arrows {\rm(}including all formal inverses of arrows{\rm)}
    are not on any forbidden cycle, the cohomological width of $\comp{X}$
    is less than or equal to $2$ {\rm(}$\comp{X}$ is the string/band complex corresponding to $\sfh${\rm)}.
\end{itemize}
\end{theorem}

Finally, we consider the Krull--Gabriel dimensions of three algebras
$A^{\CMA}$, $A$, and $A^{\CMA}/A^{\CMA}(1-\e_{\nota})A^{\CMA}$
given in the recollements $\calR(A^{\CMA}, \e_{\nota})$ and $\calR(A^{\CMA}, 1-\e_{\nota})$
in the case of $A$ being a gentle one-cycle algebra, see the following result.

\begin{corollary}[{Corollary \ref{coro:main3}}]
Let $A=\kk\Q/\I$ be a gentle one-cycle algebra. Then:
\begin{enumerate}[label=\text{\rm(\arabic*)}]
  \item $\KGdim\Ab(\Dcat^b(A)) \=< 2$ if and only if $\KGdim\Ab(\Dcat^b(A^{\CMA}))\=<2$;
  \item $\KGdim\Ab(\Dcat^b(A))=\KGdim\Ab(\Dcat^b(A^{\CMA}))$ if and only if $A\cong A^{\CMA}$;
  \item $\KGdim\Ab(\Dcat^b(A^{\CMA}/A^{\CMA}(1-\e_{\nota})A^{\CMA})) = 0$.
\end{enumerate}
\end{corollary}

\section{Gentle algebras, strings and bands} \label{background}

Gentle algebras are an important class of finite-dimensional algebras which are introduced by Assem and Skowro\'nski in \cite{AS1987}.
The indecomposable finite-dimensional modules over gentle algebras have been classified in \cite{WW1985, BR1987},
and the indecomposable objects in the derived categories of (graded) gentle algebras are described in \cite{BM2003, ALP2016}.
Base on the above results, the module categories and derived categories of gentle algebras,
skew-gentle algebras, and string algebras have garnered significant attention from numerous algebraists
and have been extensively studied, now becoming one of the primary research areas in representation theory,
including (derived) representation types \cite[etc]{Kra1991, V2001, ZH2016, Zhang2019Indecomposables},
derived categories and singularity categories \cite[etc]{AG2008, Kal2015, KY2018, KS2022, CS2023},
tilting and silting theories \cite[etc]{Pla2019, FGLZ2023, CS2023b, LiuZhou2025},
geometric models \cite[etc]{HKK2017,OPS2018,LP2020,BCS2021,QZZ2022, HZZ2023,APS2023, ZhangLiu2024},
and homological properties and recollements \cite[etc]{R1997, LGH2024, LLM2025}.

In this section, we recall the definition of gentle algebra and the description of its module category.

\subsection{Gentle pairs and gentle algebras}

Recall that a quiver is a quadruple $\Q=(\Q_0,\Q_1,\s,\t)$,
where $\Q_0$ and $\Q_1$ are two sets whose elements are called vertices and arrows, respectively,
and $\s,\t:\Q_1 \to \Q_0$ are the maps sending each arrow $a \in \Q_1$ to its starting point $\s(a)$ and ending point $\t(a)$.
Throughout this paper, we assume that $\Q$ is a finite quiver, i.e., $\Q_0$ and $\Q_1$ are finite sets,
and we adopt the convention that for two arrows $a,b \in \Q_1$ with $\t(a)=\s(b)$, their composition is written as $ab$,
Let $\kk$ be an algebraically closed field. Denote by $\kk\Q_{\>= 1}$ the ideal of the path algebra $\kk\Q$ generated by all arrows.
An admissible ideal $\I$ of $\kk\Q$ is an ideal satisfying
$(\kk\Q_{\>=1})^n \subseteq \I \subseteq (\kk\Q_{\>=1})^2$ for some $n\>= 2$.
and a bound quiver is defined as a pair $(\Q, \I)$ given by a quiver $\Q$ and an ideal $\I$.
For each arrow $a\in \Q_1$, we define that it has a \defines{formal inverse} and write it as $a^{-1}$.
Then, naturally, we can define $\Q_1^{-1}$ as the set of all formal inverses of arrows,
and the functions $\s$ and $\t$ can be extended to $\Q_1^{\pm1} := \Q_1\cup \Q_1^{-1}$
such that $\s(a)=\t(a^{-1})$ and $\t(a)=\s(a^{-1})$ for each arrow or each formal inverses of an arrow $a$
(where, $(a^{-1})^{-1}:=a$, and for each vertex $v$, the formal inverse of the path $e_v$ of length zero corresponding to $v$ is itself).

A \defines{gentle pair} is a bound quiver $(\Q,\I)$ satisfying the following four conditions:
\begin{enumerate}[label=\textrm{(G\arabic*)}]
  \item  \label{G1}
    every vertex of $\Q$ admits at most two incoming arrows and two outgoing arrows;
  \item  \label{G2}
    for arrows $a$, $b$, $c$ with $a \neq b$ and $\t(a) = \t(b) =\s(c)$, exactly one of the paths $ca$, $cb$ is in $\I$;
  \item  \label{G3}
    for arrows $a$, $b$, $c$ with $a \neq b$ and $\s(a) =\s(b) = \t(c)$, exactly one of the paths $ac$, $bc$ is in $\I$;
  \item  \label{G4}
    the ideal $\I$ is generated by paths of length two.
\end{enumerate}

\begin{definition}[{Gentle algebras \cite{AS1987}}]\rm
A gentle algebra $A$ is a finite-dimensional algebra $A=\kk\Q/\I$ whose bound quiver is a gentle pair.
\end{definition}

\begin{example}\rm \label{examp:gentle}
Let $A=\kk\Q/\I$ be an algebra given by $(\Q,\I)$, where $\Q$ is shown in \Pic \ref{fig:Examplegentle},
\begin{figure}[htbp]
  \centering
	\begin{tikzpicture}[scale=0.5]
		\foreach \x in {0,120,240}
		\draw[->][rotate around = {10+\x:(0,0)}][line width=1pt] (2,0) arc(0:100:2);
		\foreach \x in {0,120,240}
		\draw[->][rotate around = {5+\x:(0,0)}][line width=1pt] (6,0) arc(0:110:6);
		\draw (2,0) node{$1$} (-1*1,1.73*1) node{$2$} (-1*1,-1.73*1) node{$3$};
		\draw (4,0) node{$4$} (-1*2,1.73*2) node{$5$} (-1*2,-1.73*2) node{$6$};
		\draw (6,0) node{$7$} (-1*3,1.73*3) node{$8$} (-1*3,-1.73*3) node{$9$};
		\foreach \x in {0,120,240}
		\draw[->][rotate around = {\x:(0,0)}][line width=1pt] (3.7,0) -- (2.3,0);
		\foreach \x in {0,120,240}
		\draw[->][rotate around = {\x:(0,0)}][line width=1pt] (5.7,0) -- (4.3,0);
		\draw[rotate around = { 60:(0,0)}] (2.4,0) node{$a_{12}$};
		\draw[rotate around = {180:(0,0)}] (2.4,0) node{$a_{23}$};
		\draw[rotate around = {300:(0,0)}] (2.4,0) node{$a_{31}$};
		\draw[rotate around = { 60:(0,0)}] (6.4,0) node{$a_{78}$};
		\draw[rotate around = {180:(0,0)}] (6.4,0) node{$a_{89}$};
		\draw[rotate around = {300:(0,0)}] (6.4,0) node{$a_{97}$};
		\draw[rotate around = {  0:(0,0)}] (3,0.5) node{$a_{41}$};
		\draw[rotate around = {120:(0,0)}] (3,0.5) node{$a_{52}$};
		\draw[rotate around = {240:(0,0)}] (3,0.5) node{$a_{63}$};
		\draw[rotate around = {  0:(0,0)}] (5,0.5) node{$a_{74}$};
		\draw[rotate around = {120:(0,0)}] (5,0.5) node{$a_{85}$};
		\draw[rotate around = {240:(0,0)}] (5,0.5) node{$a_{96}$};
		\foreach \x in {0,120,240}
		\draw[line width=1pt][dotted]
		[rotate around = {\x:(0,0)}] (1.85,-0.5) arc(-90:-270:0.5);
		\foreach \x in {0,120,240}
		\draw[line width=1pt][dotted]
		[rotate around = {\x:(0,0)}] (5.95,-1) arc(-90:90:1);
		\foreach \x in {0,120,240}
		\draw[line width=1pt][dashed]
		[rotate around = {\x:(0,0)}] (3,0) arc(-180:0:1);
	\end{tikzpicture}
  \caption{\texttt{Bound quiver}}
  \label{fig:Examplegentle}
\end{figure}
and $\I$ is generated by the paths $a_{12}a_{23}$, $a_{23}a_{31}$, $a_{31}a_{12}$, $a_{74}a_{41}$,
$a_{85}a_{52}$, $a_{96}a_{63}$, $a_{78}a_{89}$, $a_{89}a_{97}$, and $a_{87}a_{78}$.
Then $A$ is a gentle algebra.
\end{example}

\subsection{Strings and bands}

Let $(\Q,\I)$ be a gentle pair. A \defines{walk} in $(\Q,\I)$ of length $n>0$ is a sequence $\sfw=\sfw_1\cdots \sfw_n$
in $\Q_1^{\pm 1}$ such that $\t(\sfw_{i+1})=\s(\sfw_i)$ holds for all $1\=< i<n$.
Each $\sfw_i$ is called the \defines{edges} of $\sfw$, and, in particular,
$\sfw_1$ and $\sfw_n$ are called the \defines{initial edge} and \defines{terminal edge} of $\sfw$, respectively.
A \defines{subwalk} of $\sfw$ is either a subsequence $\sfw_i\cdots \sfw_j$ of $\sfw$ ($1\=< i \=< j \=< n$)
or a trivial path $e_v$ corresponding to some vertex $v \in \{\s(\sfw_1), \ldots, \s(\sfw_n), \t(\sfw_n)\}$.
Naturally, we can define $\s(\sfw)=\s(\sfw_1)$ and $\t(\sfw)=\t(\sfw_n)$, and say that $w$ is a walk from $\s(\sfw)$ to $\t(\sfw)$.
The formal inverse of $w$ is defined as $\sfw^{-1}=\sfw_n^{-1}\cdots \sfw_1^{-1}$.
Furthermore, a walk is said to be
\begin{itemize}
  \item a \defines{trivial walk} if it is a trivial path $e_v$ with $\s(e_v)=\t(e_v)=v\in\Q_0$;
  \item a \defines{reduced walk} if it is either a trivial walk
    or a walk $\sfw  = \sfw_1\cdots \sfw_n$ such that $\sfw_{i+1}\neq \sfw_i^{-1}$ holds for all $1\=< i<n$;
  \item a \defines{string} if it is a reduced walk such that any subwalk of it does not in $\I$;
  \item a \defines{band} if it is a string with $\t(\sfw)=\s(\sfw)$ such that $\sfw$ is not a non-trivial power of any string and, for all $n\>=2$, $\sfw^n$ is a string.
\end{itemize}

Two strings $\sfs$ and $\sfs'$ are called \defines{equivalent} if one of $\sfs=\sfs'$ and $\sfs^{-1}=\sfs'$ holds.
Two bands $\sfb$ and $\sfb'$ are called \defines{equivalent} if there exists $0\=< t < \ell(\sfb)$
such that one of $\sfb=\sfb'[t]$ and $\sfb^{-1}=\sfb'[t]$ holds ($\ell$ is the map sending each string/band to its length).
Here, for each band $\sfb = b_1\cdots b_n$ of length $n$, $\sfb[t]$ ($0\=< t < n$)
represents the band $b_{1+t}b_{2+t}\cdots b_{n}b_1b_2\cdots b_{t}$.
For simplicity, we still use $[\sfs]$ (resp., $[\sfb]$) to represent
the equivalence class containing $\sfs$ (resp., $\sfb$), this not cause confusion.
The module category $\modcat(A)$ of a gentle algebra $A$ can be described by using strings and bands.

\begin{theorem}[{Butler, Ringel, Wald, and Waschb\"{u}sch \cite{WW1985, BR1987}}] \label{thm:WWBR}
Let $A=\kk\Q/\I$ be a gentle algebra and $\Str(A)$ {\rm(}resp., $\Band(A)${\rm)}
be the set of all equivalence classes of strings {\rm(}resp., bands{\rm)}.
Then there is a bijection
\[\M: \Str(A) \cup (\Band( A)\times\mathscr{J}) \to \ind(\modcat(A))\]
from the disjoint union $\Str(A) \cup (\Band(A)\times\mathscr{J})$ to
the set of all isoclasses of indecomposable $A$-modules,
where $\mathscr{J}$ is the set of all Jordan block with non-zero eigenvalue.
\end{theorem}

\begin{notation} \rm
For any string $\sfs$ or band $\sfb$ with a Jordan block $\pmb{J}$,
we define $\M(\sfs):=\M([\sfs])$ and $\M((\sfb,\pmb{J})):=\M(([\sfb], \pmb{J}))$.
\end{notation}

This theorem provides two important facts as follows:

\begin{itemize}
  \item $\ind(\modcat(A)) = \image(\M|_{\Str(A)}) \cup \image(\M|_{\Band( A)\times\mathscr{J}})$;
  \item $\image(\M|_{\Str(A)}) \cap \image(\M|_{\Band(A)\times\mathscr{J}}) = \varnothing$.
\end{itemize}

In \cite{BR1987}, each indecomposable module lying in $\image(\M|_{\Str(A)})$ is called a \defines{string module},
and each indecomposable module lying in $\image(\M|_{\Band(A)\times\mathscr{J}})$ is called a \defines{band module}.

\begin{remark}\label{rmk:WWBR} \rm
By Theorem \ref{thm:WWBR}, any indecomposable $A$-module $M$ is either a string module or a band module.
If $M$ is a band module, then it is well-know that the projective dimension and injective dimension of $M$ are $1$,
and if $M$ is a string module, then its projective dimension and injective dimension can be described by using
forbidden paths on the bound quiver $(\Q,\I)$ of $A$.
\end{remark}

\section{Quasi-tilted gentle algebras}

We use ``$\pdim$'', ``$\idim$'', ``$\gldim$'', and ``$\findim$'' to represent
projective dimension, injective dimension, global dimension, and finitistic dimension, respectively.
A finite-dimensional algebra $A$ is said to be a \defines{quasi-tilted algebra} if
its global dimension $\gldim A$ is less than or equal to $2$,
and for each $A$-module $M$, either $\pdim M \=< 1$ of $\idim M \=< 1$ holds, see \cite{HRS1994}.
This homological condition characterizes the symmetry of its derived category
and connects tilted-algebras and hereditary algebras.
Therefore, it is an important problem to determine the criteria for whether a given algebra is quasi-tilted.

\subsection{Homological bounds}

The property of being a quasi-tilted algebra is limited by the global dimension of the algebra.
Thus, if one needs to determine whether an algebra is quasi-tilted,
it is also necessary to determine the homological dimensions
(precisely, to determine the projective dimensions and injective dimensions)
of all modules over this algebra. In particular, \cite[Theorem 3.7]{LMXZ2025}
shows that it is necessary to provide conditions for determining
when the homological dimensions of modules are less than or equal to $2$
in the case where the algebra is gentle. Now, we recall some concepts in \cite{AG2008}
and some results in \cite{LMXZ2025}, which are important for this paper.

\begin{definition}[{\rm\cite[Section 2]{AG2008}}] \rm \label{def:forb.path}
A \defines{forbidden path} $F=a_1\cdots a_n$ ($n\in\NN$) is a path on $\Q$ such that
one of the following conditions holds:
\begin{itemize}
  \item $n=0$ (in this case, $F=e_v$ is a path of length zero corresponding to some vertex $v$),
    and $v$ is a relational vertex;
  \item $n=1$, i.e., $F$ is a path of length $1$;
  \item $n\>= 2$, and for any $1\=< i < n$, we have $a_ia_{i+1}\in\I$.
\end{itemize}

Furthermore, if for any arrow $\alpha$ with $\t(\alpha)=\s(a_1)$ (resp., $\t(a_n)=\s(\alpha)$),
we have $\alpha a_1\notin\I$ (resp., $a_n\alpha \notin \I$),
then we call $F$ is a \defines{left {\rm(}resp., right{\rm)} maximal forbidden path};
if $F$ is both left maximal and right maximal, then we call it is a \defines{maximal forbidden path};
if $F$ is an oriented cycle and $a_na_1\in\I$, then we call it is a \defines{forbidden cycle}
(or a \defines{full-relational oriented cycle}).
\end{definition}

Obviously, all vertices on a forbidden path are relational vertices.
However, the definition of relational vertices can be further refined to a more specific reduced walk,
i.e., we say a vertex $v$ on a reduced walk $\sfw=w_1 \cdots w_n$ ($w_1,\ldots,w_n\in\Q_1^{\pm1}$) is a
\defines{relational vertex on $\sfw$} if one of following conditions holds:
\begin{itemize}
  \item there are two $w_i$, $w_{i+1}\in\Q_1$ such that $v$, $\t(w_i)$, and $\s(w_{i+1})$ coincide;
  \item there are two $w_i$, $w_{i+1}\in\Q_1^{-1}$ such that $v$, $\t(w_{i+1}^{-1})$, and $\s(w_i^{-1})$ coincide.
\end{itemize}

\begin{definition}[{\rm\cite[Section 3]{LMXZ2025}}] \rm
A \defines{strong source} (resp., strong sink) $v$ of a quiver $\Q$ is a source (resp., sink) of this quiver
such that there at most one arrow $\alpha$ with $\s(\alpha)=v$ (resp., $\t(\alpha)=v$).
\end{definition}

\begin{theorem}[{\cite[Theorem 3.7]{LMXZ2025}}] \label{thm:LMXZ}
Let $A$ be a gentle algebra with finite global dimension.
If the starting {\rm(}or ending{\rm)} vertices of all maximal forbidden paths of length $\>=2$
are strong sources {\rm(}or strong sinks{\rm)},
then \[ \sup_{M\in\ind(\modcat(A))}(\pdim M + \idim M) \=< 2~\gldim A -1. \]
\end{theorem}

\begin{remark}\label{rmk:LMXZ} \rm
The supremum given in Theorem \ref{thm:LMXZ} is called the \defines{homological bounds}.
By this theorem, if $A$ is gentle, then the definition of a quasi-tilted algebra has an equivalent statement,
namely, $\gldim A \=< 2$, and for all non-projective indecomposable modules
and non-injective indecomposable modules $M$, we have $\pdim M  + \idim M \=< 3$,
see \cite[Corollaries 4.2 and 4.3]{LMXZ2025}.
\end{remark}

\subsection{Modules with projective/injective dimensions two}

Usually, if we need to compute the projective dimension of a string module $\M(\sfs)$ corresponding to a string $\sfs$,
we need to consider four important forbidden paths $F_{\mathrm{lu}}$, $F_{\mathrm{ru}}$, $F_{\mathrm{ld}}$, $F_{\mathrm{rd}}$
shown in \Pic \ref{fig:four forb paths} (their lengths may be zero),
\begin{figure}[htbp]
\centering
\[
\xymatrix{
\bullet \ar@{~>}[rd]^{F_{\mathrm{lu}}} &&&& \bullet \ar@{~>}[ld]_{F_{\mathrm{ru}}}\\
& v \ar@{~>}[ld]^{F_{\mathrm{ld}}} \ar@{~}[rr]^{\sfs}& & w \ar@{~>}[rd]_{F_{\mathrm{rd}}} & \\
\bullet & & & & \bullet
}
\]
\caption{\texttt{The string $\sfs$ and the four forbidden paths associated with it}}
\label{fig:four forb paths}
\end{figure}
where $v$, as a vertex on the reduced walk $F_{\mathrm{lu}}\sfs$ (resp., $F_{\mathrm{ld}}^{-1}\sfs$), is not a relational vertex;
and $w$, as a vertex on the reduced walk $\sfs F_{\mathrm{ru}}^{-1}$ (resp., $\sfs F_{\mathrm{rd}}$), is not a relational vertex;
$F_{\mathrm{lu}}$ and $F_{\mathrm{ru}}$ are left maximal;
and $F_{\mathrm{ld}}$ and $F_{\mathrm{rd}}$ are right maximal.
We call the above four forbidden paths are \defines{forbidden paths with respect to $\sfs$}.

\begin{lemma} \label{lemma:dim=2}
A string module $\M(\sfs)$ corresponding to a string $\sfs$ in a gentle pair $(\Q,\I)$ has both a projective dimension of $2$
and an injective dimension of $2$ if and only if the following conditions \ref{lemma:dim=2 (A)}, \ref{lemma:dim=2 (B)} and \ref{lemma:dim=2 (C)} hold.
\begin{enumerate}[label=\text{\rm(\Alph*)}]
  \item \label{lemma:dim=2 (A)}
    The lengths of all forbidden paths with respect to $\sfs$ are less than or equal to $2$.
  \item \label{lemma:dim=2 (B)}
    There is at least one left maximal forbidden path $F$ with respect to $\sfs$ satisfying $\ell(F)=2$.
  \item \label{lemma:dim=2 (C)}
    There is at least one right maximal forbidden path $F$ with respect to $\sfs$ satisfying $\ell(F)=2$.
\end{enumerate}
\end{lemma}

%

\begin{proof}
Sufficiency (``if'' part):
Assume that $\sfs$ is a string such that $\pdim \M(\str)=\idim \M(\str)=2$, and, as a general situation,
$\sfs$ has the form shown in \Pic \ref{fig:four forb paths},
where $F_{\mathrm{lu}}$ and $F_{\mathrm{ru}}$ are left maximal forbidden paths of length $\=<2$,
and $F_{\mathrm{ld}}$ and $F_{\mathrm{rd}}$ are right maximal forbidden paths of length $\=<2$.
Next, we show
\begin{align}
  \pdim\M(\sfs) &= \max\{\ell(F_{\mathrm{ld}}), \ell(F_{\mathrm{rd}})\} \label{formula:lemma:dim=2 pdim} \\
  \idim\M(\sfs) &= \max\{\ell(F_{\mathrm{lu}}), \ell(F_{\mathrm{ru}})\} \label{formula:lemma:dim=2 idim}
\end{align}
in this proof.

If $\ell(F_{\mathrm{ld}})=\ell(F_{\mathrm{ld}})=2$, then we assume $F_{\mathrm{ld}}=x_1x_2$,
and $F_{\mathrm{rd}}=y_1y_2$ ($x_1, x_2, y_1, y_2\in\Q_1$).
Then, by \cite[Lemma 3.4]{ZhangLiu2024}, the 1-syzygy $\Omega_1(\M(\sfs))$ of $\sfs$ is isomorphic to
\begin{align}\label{formula:lemma:dim=2 1-syzygy}
  \Omega_1(\M(\sfs)) \cong L_1 \oplus P \oplus R_1,
\end{align}
where $\top L_1 = S(\t(x_1))$, $\top R_1 = S(\t(y_1))$, and $P$ is a projective $A$-module
(for any $v\in\Q_0$, $S(v)$ is the simple module corresponding to $v$).
Thus, the projective cover of $\Omega_1(\M(\sfs))$ is of the form
$P(\t(x_1)) \oplus P \oplus P(\t(y_1)) \to \Omega_1(\M(\sfs))$,
it follows that the 2-syzygy of $\M(\sfs)$ is
\begin{align}\label{formula:lemma:dim=2 2-syzygy}
  \Omega_2(\M(\sfs)) \cong L_2\oplus R_2,
\end{align}
where $\top L_2 \cong S(\t(x_2))$, $\top R_2 \cong S(\t(y_2))$.
Note that $A$ is gentle and $F_2, G_2$ are right maximal,
then $L_2$ and $R_2$ are projective. So, $\pdim \M(\sfs)= 2 =
\ell(F_{\mathrm{ld}}) = \ell(G_{\mathrm{rd}})$ in this case.

If $\ell(F_{\mathrm{ld}})=2 > \ell(F_{\mathrm{rd}}) =1$,
then $\Omega_1(\M(\sfs))$ is still of the form given in (\ref{formula:lemma:dim=2 1-syzygy}).
However, one can check that $R_1$ is projective (or equivalently, one can check that
$R_2$, the direct summand of $\Omega_2(\M(\sfs))$ shown in \ref{formula:lemma:dim=2 2-syzygy}, is zero).
So, $\pdim \M(\sfs)= 2 =\ell(F_{\mathrm{ld}}) > \ell(G_{\mathrm{rd}}) = 1$.
Similarly, if $\ell(F_{\mathrm{ld}})=1 < \ell(F_{\mathrm{rd}}) =2$,
we have $\pdim \M(\sfs)= 2 =\ell(F_{\mathrm{rd}}) > \ell(G_{\mathrm{ld}}) = 1$.
Furthermore, all cases satisfying $\ell(F_{\mathrm{ld}})\=<1$ and $\ell(F_{\mathrm{rd}}) \=<1$
can be computed in using similar way. Therefore, we have (\ref{formula:lemma:dim=2 pdim}).
In a dual way, we can prove (\ref{formula:lemma:dim=2 idim}).

The condition \ref{lemma:dim=2 (A)} admits $\pdim\M(s)\=<2$ and admits $\idim\M(s)\=<2$,
condition \ref{lemma:dim=2 (B)} admits $\idim\M(s)\>=2$,
and condition \ref{lemma:dim=2 (C)} admits $\pdim\M(s)\>=2$,
then the sufficiency holds.

Necessity (``only if'' part):
if one of $\ell(F_{\mathrm{ld}})\>= 3$ and $\ell(F_{\mathrm{rd}})\>= 3$ holds,
then consider the case of $\ell(F_{\mathrm{ld}})\>= 3$,
we can write $F_{\mathrm{ld}}$ is of the form $F_{\mathrm{ld}}=x_1x_2x_3\ldots x_n$ ($n\>=3$).
In this case, the 1-syzygy $\Omega_1(\M(\sfs))$ can be written as
(\ref{formula:lemma:dim=2 1-syzygy}) such that $L_1$ is non-zero,
and the 2-syzygy $\Omega_2(\M(\sfs))$ can be written as
(\ref{formula:lemma:dim=2 2-syzygy}) such that $L_2$ is non-zero.
Furthermore, $L_1$ and $L_2$ are not projective.
It follows that $\Omega_3(\M(\sfs))$ has a non-zero direct summand $L_3$,
which is given by the vertex $\t(x_3)$ of the forbidden path $F_{\mathrm{ld}}$,
to be more precise, $\top L_3 = S(\t(x_3)) \ne 0$.
Thus, we have $\pdim \M(\s)\>= 3$, a contradiction.
Then $\ell(F_{\mathrm{ld}})\=<2$.
We can prove the lengths of other three forbidden paths respect to $\sfs$ are
less than or equal to $2$, i.e., we obtain \ref{lemma:dim=2 (A)}.
On the other hand:
if the lengths of all left maximal forbidden paths respect to $\sfs$ are less than or equal to $1$,
then $\idim\M(s)\=< 1$ holds by using (\ref{formula:lemma:dim=2 idim}),
we obtain a contradiction. Then \ref{lemma:dim=2 (B)} holds.
One can check that \ref{lemma:dim=2 (C)} holds in a dual way and we complete the proof of this lemma.
\end{proof}

Notice that the condition \ref{lemma:dim=2 (A)} in Lemma \ref{lemma:dim=2} admits
$\pdim\M(\sfs)\=< 2$ and $\idim\M(\sfs)\=< 2$.
Thus, a direct corollary of Lemma \ref{lemma:dim=2} is the following result.

\begin{corollary}
Keep the notations from Lemma \ref{lemma:dim=2}.
\begin{itemize}
  \item[\rm(1)] Condition \ref{lemma:dim=2 (B)} given in Lemma \ref{lemma:dim=2} holds if and only if $\idim\M(\sfs)\>=2$;
  \item[\rm(2)] Condition \ref{lemma:dim=2 (C)} given in Lemma \ref{lemma:dim=2} holds if and only if $\pdim\M(\sfs)\>=2$.
\end{itemize}
\end{corollary}

\subsection{CM-Auslander algebras of gentle algebras}

A \defines{Gorenstein-projective module} ($=$ G-projective module for short) $G$ over an algebra $A$,
a generalized projective module introduced by Auslander and Bridger \cite{AB1969},
is an $A$-module with a $\Hom_A(-,A)$-exact complete projective pre-resolution
\[ \pmb{P}_{\bullet} := ~
\xymatrix{ \cdots \ar[r]^{p_{-2}} & P_{-1} \ar[r]^{p_{-1}} &
  P_0 \ar[r]^{p_0} & P_1 \ar[r]^{p_1} & P_2 \ar[r]^{p_2} & \cdots,
}\]
i.e., there is an exact sequence as above such that
$G\cong \image(p_{-1})$ and $\Hom_A(\pmb{P}_{\bullet}, A)$ is exact.
The dual of Gorenstein-projective module is Gorenstein-injective module
($=$ G-injective module for short), see \cite[Definition 2.1]{EJ1995}.
Let $\Gproj(A)$ be the set of all G-projective $A$-modules (up to isomorphism)
and $\ind(\Gproj(A))$ be the set of all indecomposable G-projective $A$-modules (up to isomorphism).
The \defines{Cohen--Macaulay--Auslander algebra} ($=$ CM--Auslander algebra) of $A$
is defined as the following endomorphism algebra
\[
\bigg(\End_A \bigg(
  \bigoplus_{G \in \ind(\Gproj(A))
  }
  G
\bigg)\bigg)^{\op},
\]
see \cite[Section 6.2]{B2005}. Here, for a ring/algebra $\itLamb$,
``$\itLamb^{\op}$'' is its \defines{opposite ring/algebra}.

In \cite{GR2005}, Gei\ss and Reiten have proved that all gentle algebras are Gorenstein.
Motivated by the classical definition of the Auslander algebra, Beligiannis introduced
the Cohen--Macaulay Auslander algebra (= CM--Auslander algebra) of an algebra,
which is defined as the endomorphism algebra of the direct sum of all Gorenstein-projective modules which are pairwise not isomorphic~\cite{B2005,B2011}.
Thus, it is naturally and important to characterize the relations and homological properties
between algebras and their CM--Auslander algebras.
In \cite{CL2017, CL2019}, Chen and Lu described the
CM--Auslander algebras of gentle and skew-gentle algebras,
building on the works of Kalck in \cite{Kal2015}.
Now, we recall how to compute the CM--Auslander algebra of a gentle algebra.

Let $A$ be a gentle algebra and $(\Q,\I)$ its bound quiver.
Let $(\Q^{\CMA},\I^{\CMA})$ be the bound quiver obtained by the following steps.
\begin{enumerate}[label=\text{\rm Step~\arabic*.}]
  \item there is a bijection $\mathfrak{v}: \Q^{\CMA}_0 \mathop{\rightarrow}\limits^{1-1} \Q_0 \cup (\ind(\Gproj(A))\backslash\ind(\proj(A)))$
    \footnote{Here, $\ind(\proj(A))$ is the set of all indecomposable projective $A$-modules (up to isomorphism),
    and each element $v=\mathfrak{v}^{-1}(G)$ in $\Q_0^{\CMA}$
    is said to be the vertex corresponding to the indecomposable non-projective G-projective module $G$.}.
  \item For any forbidden cycle $\C=c_1\cdots c_n$ ($\s(c_i)=v_i$),
    define $c_i^-$ and $c_i^+$ the two arrows obtained by $c_i$ as follows:
    \begin{itemize}
      \item $c_i^-$ is an arrow from $\s(c_i)$ to the G-projective module $c_i A \in \ind(\Gproj(A))$,
      \item $c_i^+$ is an arrow from the G-projective module $c_i A \in \ind(\Gproj(A))$ to $\t(c_i)$.
    \end{itemize}
    Here, Kalck proved that $c_i A$ is an indecomposable G-projective module,
    see \cite[Theorem 2.5 (a)]{Kal2015}. Define
    \footnote{This definition induces two functions $\s, \t: \Q^{\CMA}_1 \to \Q^{\CMA}_0$
    such that $\Q^{\CMA}=(\Q^{\CMA}_1$, $\Q^{\CMA}_0$, $\s^{\CMA}$, $\t^{\CMA})$ is a quiver.}
    \begin{align*}
        \Q_1^{\CMA}
      & := \Q_1\backslash\{ c\in\Q_1 \mid c \text{ is an arrow on a forbidden cycle} \} \\
      & \cup \{c^- \mid c \text{ is an arrow on a forbidden cycle} \} \\
      & \cup \{c^+ \mid c \text{ is an arrow on a forbidden cycle} \}.
    \end{align*}

  \item $\I^{\CMA}$ is the ideal of $\kk\Q^{\CMA}$ generated by the following two types of paths:
    \begin{itemize}
      \item the path $ab$ in $\I$ such that $a$ and $b$ are three arrows not on any forbidden cycle;
      \item the path $ac^{-}$ such that $ac\in\I$;
      \item the path $c^{+}b$ such that $cb\in\I$.
    \end{itemize}
\end{enumerate}

\begin{theorem}[{Chen--Lu \cite[Theorem 3.5]{CL2019}}] \label{thm:CL}
Let $A=\kk\Q/\I$ be a gentle algebra and $A^{\CMA}$ be its CM--Auslander algebra.
Then $A^{\CMA}$ is also gentle and $A^{\CMA} \cong \kk\Q^{\CMA}/\I^{\CMA}$.
\end{theorem}

The vertices in $\Q^{\CMA}_0$ corresponding to indecomposable non-projective G-projective module $c_{ij}A$
is written as $v_{ij}$. Here, $c_{ij}$ is an arrow on a forbidden cycle $\C_i=c_{i1}\cdots c_{in}$ of $(\Q,\I)$
($\C_i$ represents the $i$-th forbidden cycle of $\Q$).
In particular, if $(\Q,\I)$ has only one forbidden cycle $\C=c_1\cdots c_n$,
then we denote $\C_1 = c_{11}\cdots c_{1n}$ by this forbidden cycle.

\begin{remark} \rm
The CM--Auslander algebra of a skew-gentle algebra is described in \cite{CL2017},
and the CM--Auslander algebra of a string algebra is described in \cite{LZhang2023CM-Auslander}.
\end{remark}

\begin{example}\label{examp:CMA of gentle} \rm
Consider the gentle algebra $A=\kk\Q/\I$ given in Example \ref{examp:gentle},
then the bound quiver of its CM--Auslander algebra is shown in \Pic \ref{fig:CMA}.
\def\lengthone{0.1}
\def\lengthtwo{0.2}
\begin{figure}[htbp]
  \centering
\begin{tikzpicture}[scale=0.5]
\foreach \x in {0,60,120,180,240,300}
\draw[->][rotate around = {10+\x:(0,0)}][line width=1pt] (2,0) arc(0:40:2);
\foreach \x in {0,60,120,180,240,300}
\draw[->][rotate around = {5+\x:(0,0)}][line width=1pt] (6,0) arc(0:50:6);
\draw (2,0) node{$1$} (-1*1,1.73*1) node{$2$} (-1*1,-1.73*1) node{$3$};
\draw (4,0) node{$4$} (-1*2,1.73*2) node{$5$} (-1*2,-1.73*2) node{$6$};
\draw (6,0) node{$7$} (-1*3,1.73*3) node{$8$} (-1*3,-1.73*3) node{$9$};
\foreach \x in {0,120,240}
\draw[->][rotate around = {\x:(0,0)}][line width=1pt] (3.7,0) -- (2.3,0);
\foreach \x in {0,120,240}
\draw[->][rotate around = {\x:(0,0)}][line width=1pt] (5.7,0) -- (4.3,0);
\draw[rotate around = { 60-30:(0,0)}] (2.4+\lengthone,0) node{$a_{12}^{-}$};
\draw[rotate around = {180-30:(0,0)}] (2.4+\lengthone,0) node{$a_{23}^{-}$};
\draw[rotate around = {300-30:(0,0)}] (2.4+\lengthone,0) node{$a_{31}^{-}$};
\draw[rotate around = { 60-30:(0,0)}] (6.4+\lengthtwo,0) node{$a_{78}^{-}$};
\draw[rotate around = {180-30:(0,0)}] (6.4+\lengthtwo,0) node{$a_{89}^{-}$};
\draw[rotate around = {300-30:(0,0)}] (6.4+\lengthtwo,0) node{$a_{97}^{-}$};
\draw[rotate around = { 60+30:(0,0)}] (2.4+\lengthone,0) node{$a_{12}^{+}$};
\draw[rotate around = {180+30:(0,0)}] (2.4+\lengthone,0) node{$a_{23}^{+}$};
\draw[rotate around = {300+30:(0,0)}] (2.4+\lengthone,0) node{$a_{31}^{+}$};
\draw[rotate around = { 60+30:(0,0)}] (6.4+\lengthtwo,0) node{$a_{78}^{+}$};
\draw[rotate around = {180+30:(0,0)}] (6.4+\lengthtwo,0) node{$a_{89}^{+}$};
\draw[rotate around = {300+30:(0,0)}] (6.4+\lengthtwo,0) node{$a_{97}^{+}$};
\draw[rotate around = {  0:(0,0)}] (3,0.5) node{$a_{41}$};
\draw[rotate around = {120:(0,0)}] (3,0.5) node{$a_{52}$};
\draw[rotate around = {240:(0,0)}] (3,0.5) node{$a_{63}$};
\draw[rotate around = {  0:(0,0)}] (5,0.5) node{$a_{74}$};
\draw[rotate around = {120:(0,0)}] (5,0.5) node{$a_{85}$};
\draw[rotate around = {240:(0,0)}] (5,0.5) node{$a_{96}$};
\foreach \x in {0,120,240}
\draw[line width=1pt][dotted][red]
[rotate around = {\x:(0,0)}] (1.85,-0.5) arc(-90:-270:0.5);
\foreach \x in {0,120,240}
\draw[line width=1pt][dotted][blue]
[rotate around = {\x:(0,0)}] (5.95,-1) arc(-90:90:1);
\foreach \x in {0,120,240}
\draw[line width=1pt][dashed][violet]
[rotate around = {\x:(0,0)}] (3,0) arc(-180:0:1);
\draw[rotate around = { 60:(0,0)}] (2,0) node{\footnotesize$G_{12}$} (-1*1,1.73*1);
\draw[rotate around = {180:(0,0)}] (2,0) node{\footnotesize$G_{23}$} (-1*1,1.73*1);
\draw[rotate around = {300:(0,0)}] (2,0) node{\footnotesize$G_{31}$} (-1*1,1.73*1);
\draw[rotate around = { 60:(0,0)}] (6,0) node{\footnotesize$G_{78}$} (-1*1,1.73*1);
\draw[rotate around = {180:(0,0)}] (6,0) node{\footnotesize$G_{89}$} (-1*1,1.73*1);
\draw[rotate around = {300:(0,0)}] (6,0) node{\footnotesize$G_{97}$} (-1*1,1.73*1);
\end{tikzpicture}
  \caption{\texttt{The CM--Auslander algebra of the gentle algebra given in Example \ref{examp:gentle}}}
  \label{fig:CMA}
\end{figure}
\end{example}

For any gentle algebra $A=\kk\Q/\I$ with forbidden cycles
$\C_1=c_{i1}\cdots c_{in_1}$, $\ldots$, $\C_t = c_{t1}\cdots c_{tn_t}$,
we define $e_{ij}$ the idempotent of $A^{\CMA}$ corresponding to the vertex $v_{ij}$
($=\mathfrak{v}^{-1}(c_{ij}A)$), and define
\begin{align}\label{formula:idempotent}
  \e_{\nota} := 1-\sum_{i=1}^t\sum_{j=1}^{n_i}e_{ij} = \sum_{v\in\Q_0} \e_v
\end{align}
in this paper (if $(\Q,\I)$ contains no forbidden cycle, then $\e_{\nota}$ is $1$).

\begin{lemma} \label{lemm:recollement CMA}
Let $A=\kk\Q/\I$ be a gentle algebra. Then there is a recollement
\begin{align}
\calR(C,\e_{\nota}):= \
\xymatrix@C=2cm{
  \modcat(\bC) \ar[r]^{\emb_{\nota}}
& \modcat(C) \ar[r]^{\res_{\e_{\nota}}}
  \ar@/_1.5pc/[l]_{-\otimes_{C} \bC}
  \ar@/^1.5pc/[l]^{\Hom_{C}(\bC, -)}
& \modcat(A).
  \ar@/_1.5pc/[l]_{-\otimes_A \e_{\nota}C}
  \ar@/^1.5pc/[l]^{\Hom_A(C\e_{\nota},-)}
}
\end{align}
of $\modcat(\bC)$, $\modcat(C)$, and $\modcat(A)$,
where $C:=A^{\CMA}$ is the CM--Auslander algebra of $A$,
and $\bC:= C/C\e_{\star} C$ is $C$ by modulo $\e_{\star}$.
\end{lemma}

\begin{proof}
By Theorem \ref{thm:CL}, we have
\begin{align*}
C \cong \kk\Q^{\CMA}_0 + \sum_{u\>= 1}\kk\Q^{\CMA}_u
  = \sum_{i=1}^t\sum_{j=1}^{n_i}e_{ij} + \sum_{v\in\Q_0} \kk\e_v + \sum_{u\>= 1}\kk\Q^{\CMA}_u.
\end{align*}
It follows
\begin{align*}
\e_{\nota}C\e_{\nota}
& \cong \sum_{v\in\Q_0} \kk\e_v + \e_{\nota} \bigg(\sum_{u\>= 1}\kk\Q^{\CMA}_u \bigg)\e_{\nota}\\
& = \sum_{v\in\Q_0} \kk\e_v + \mathrm{span}_{\kk}
    \{\text{paths through $e_v$} \mid v (\in\Q_0) \text{ is not on any } \C_i\} \\
& = \sum_{v\in\Q_0} \kk\e_v + \sum_{u\>= 1}\kk\Q_u = A.
\end{align*}
Then the recollement (\ref{recollement}) admits the existence of $\calR(C,\e_{\nota})$ by \cite[Example 2.7]{Psa2014}.
\end{proof}

\begin{remark} \rm
The algebra $\bC$ given in Lemma \ref{lemm:recollement CMA} is semi-simple,
whose quiver is obtained by all vertices of the form $\mathfrak{v}^{-1}(cA)$.
Here, all $cA$ are indecomposable non-projective G-projective $A$-modules.
\end{remark}

\begin{proposition} \label{prop:recoll CMA}
Let $A=\kk\Q/\I$ be a gentle algebra and $C$ its CM--Auslander algebra.
Keep the notations from Lemma \ref{lemm:recollement CMA},
then for any $G \in \ind(\Gproj(A))\backslash\ind(\proj(A))$, the following statements hold.
\begin{itemize}
  \item[\rm(1)]
    $\res_{\e_{\nota}}(P(v)_C)$ is G-projective, where $v=\mathfrak{v}^{-1}(G)$,
    and $P(v)_C:=e_vC$ is the indecomposable projective $C$-module corresponding to $v$.
  \item[\rm(2)]
    $G \otimes_A \e_{\nota}C$ is not projective.
    Furthermore, $G \otimes_A \e_{\nota}C$ is not G-projective,
    but $\res_{\e_{\nota}}(G \otimes_A \e_{\nota}C)$ is G-projective.
\end{itemize}
\end{proposition}

\begin{proof}
(1) Since $G$ is a non-projective indecomposable G-projective module,
then there is an arrow $\alpha=c_{ij}$ on some forbidden cycle $\C_i$ of $(\Q,\I)$
such that $G\cong \alpha A$, where $\ell(\C_i)=n_i$, see \cite[Theorem 2.5 (a)]{Kal2015}.
Thus, by Theorem \ref{thm:CL}, we obtain a correspondence
\begin{align*}
  \res_{\e_{\nota}}(P(v)_C) & = e_vC\e_{\nota} = \{e_vc\e_{\nota}\mid c\in C\} \cong \{\alpha^+\} \\
& \mathop{\rightarrow}^{\sigma} \{\alpha^+c\e_{\nota}\mid c\in C\} =: \alpha^+C\e_{\nota},
\end{align*}
which sends each $e_v\wp=\wp \in e_vC\e_{\nota}$ to $\alpha^+\wp'\in \alpha^+C\e_{\nota}$.
Here, $\wp$ is an arbitrary path in $(\Q^{\CMA}, \I^{\CMA})$
starting with $v\in\Q^{\CMA}_0$ and ending with a vertex in $\mathfrak{v}^{-1}(\Q_0)$.
Since $v$ is a vertex such that the number of arrows ending with $v$
and that of arrows starting with $v$ both are $1$,
i.e., $\t^{-1}(v)=\{\alpha^-\}$ and $\s^{-1}(v)=\{\alpha^+\}$ hold,
then we have $\wp=\alpha^+\wp'$ in $(\Q^{\CMA}, \I^{\CMA})$.
Thus, $\sigma:\wp\mapsto \alpha^+\wp'=\wp$ is a bijection.

Moreover, $\eta: \alpha^+C\e_{\nota}\to \alpha^-\alpha^+\cdot (e_{\t(\alpha)}\e_{\nota})\cdot C\e_{\nota}, ~
 \alpha^+\wp' \mapsto \alpha^-\alpha^+\wp'=\alpha\wp'$
is a bijection since $\t^{-1}(v)=\{\alpha^-\}$ contains only one arrow. Then, by
$\alpha^-\alpha^+\cdot (e_{\t(\alpha)}\e_{\nota})\cdot C\e_{\nota}
   = \alpha \e_{\nota}C\e_{\nota} = \alpha A$,
we have a bijection
\[ \eta\sigma:\res_{\e_{\nota}}(P(v)_C) \to \alpha A, \]
and one can check that it is an isomorphism of $A$-modules.
Thus, $\res_{\e_{\nota}}(P(v)_C)\cong G$ as required.

(2) First, we have
\begin{align*}
  \res_{\e_{\nota}}(M \otimes_A \e_{\nota}C)
 = (M \otimes_A \e_{\nota}C)\e_{\nota}
  \cong M \otimes_A (\e_{\nota}C\e_{\nota})
 = M \otimes_A A \cong M
\end{align*}
for all $A$-module $M$. It admits that $\res_{\e_{\nota}}(G \otimes_A \e_{\nota}C) \cong G$ is G-projective as required.

Next, we show $G \otimes_A \e_{\nota}C$ is not projective.
If $\alpha A \otimes_A \e_{\nota}C$ ($\cong \alpha^{-}\alpha^{+}C \in\modcat(C)$) is projective,
then for the projective cover $p: e_{\t(\alpha)}C \to S(\t(\alpha))$ of the simple module corresponding to $\t(\alpha)$
and the homomorphism $f: \alpha^{-}\alpha^{+}C \to S(\t(\alpha))$ sending each paths $\alpha^{-}\alpha^{+}q$
($q$ is an arbitrary path with $\s(q)=\t(\alpha)$) to $q+\rad(e_{\t(\alpha)}C)$,
there must be a homomorphism $h:\alpha^{-}\alpha^{+}C \to e_{\t(\alpha)}C$ such that $f=ph$,
i.e., we have the following diagram
\[\xymatrix{
& \alpha^{-}\alpha^{+}C \ar[d]^{f} \ar[ld]_{h} \\
e_{\t(\alpha)}C \ar[r]_{p}& S(\t(\alpha))
}\]
commutate. Then $e_{\t(\alpha)}C$ is a direct summand of $\alpha^{-}\alpha^{+}C$.
Since $\alpha^{-}\alpha^{+}C$ is indecomposable, we obtain an isomorphism
\begin{align}\label{prop:recoll CMA:iso}
     e_{\t(\alpha)}C = e_{\t(c_{ij})}C
  & \mathop{\cong}^{\tau} \alpha^{-}\alpha^{+}C = c_{ij}^{-}c_{ij}^{+}C\\
    e_{\t(c_{ij})}q & \mapsto c_{ij}^{-}c_{ij}^{+} q. \nonumber
\end{align}
Consider the arrow $c_{i\overline{j+1}}^{+} \in \Q^{\CMA}_1$
($\overline{x}$ lies in $\{1,2,\cdots,n_i\}$ satisfying $\overline{x}\equiv x~(\bmod n_i)$),
we know that this arrow, as a path of length one on $(\Q^{\CMA},\I^{\CMA})$,
is an element in $e_{\t(c_{ij})}C$, then we have
$c_{i\overline{j+1}}^+ \in e_{\t(c_{ij})}C\e_{\t(c_{i\overline{j+1}}^{+})}$.
Thus, the isomorphism (\ref{prop:recoll CMA:iso}) induces an isomorphism
\begin{align*}
 \tau: e_{\t(c_{ij})}C\e_{\t(c_{i\overline{j+1}}^{+})} \to c_{ij}^{-}c_{ij}^{+}C\e_{\t(c_{i\overline{j+1}}^{+})}
\end{align*}
satisfying $\tau(c_{i\overline{j+1}}^+) = \tau(e_{\t(c_{ij})}c_{i\overline{j+1}}^+)
= c_{ij}^{-}c_{ij}^{+}c_{i\overline{j+1}}^+ = c_{ij}^{-}\cdot (c_{ij}^{+}c_{i\overline{j+1}}^+)
= c_{ij}^{-}\cdot 0 = 0$.
It follows that $c_{i\overline{j+1}}^+=0$, this is a contradiction.
Therefore, $G \otimes_A \e_{\nota}C$ is non-projective.
Notice that $\gldim C<\infty$, then $\Gproj(C)=\proj(C)$.
Thus, $G \otimes_A \e_{\nota}C$ is not G-projective as required.
\end{proof}

%


%

\begin{definition}\rm \label{def:forb.mod.}
We call a string module over a gentle algebra \defines{forbidden}
if  the corresponding string contains at least one arrow lying on a forbidden cycle.
Otherwise, we call it \defines{non-forbidden}.
\end{definition}

Clearly, all simple modules are non-forbidden.

\begin{lemma} \label{lemm:forbcyc}
If a forbidden path $F=a_1\cdots c_n$ in a gentle pair $(\Q,\I)$ has an arrow $a_i$ that is not on any forbidden cycle of $(\Q,\I)$, then none of the arrows in $F$ lie on any forbidden cycle.
\end{lemma}

\begin{proof}
If $n=1$, we are done. Now we assume $n\geqslant 2$.
If $i\geqslant 2$, and $a_i$ is an arrow on a forbidden cycle $\C=c_1\cdots c_m$,
then there is a $c_j$ such that $a_i$ and $c_j$ coincide.
In this case, $a_{i-1}$ must be an arrow on $\C$.
Otherwise, i.e., $a_{i-1}\neq c_{\overline{j-1}}$
(if $j-1=0$, we define $\overline{j-1}=m$ in this proof),
then $a_{i-1}$ and $c_{\overline{j-1}}$ are two arrows ending with $\s(a_i)=\s(c_j)$.
Since $(\Q,\I)$ is a gentle pair, at most one of $a_{i-1}a_i\in\I$ and $c_{\overline{j-1}}c_j$ holds (see \ref{G2}),
this is a contradiction. Therefore, $a_{i-1}$ is not on $\C$.
Then all arrows $a_1$, $\ldots$, $a_i$ are not on $\C$ by induction.
If $i\leqslant n$, and $a_i$ is an arrow on $\C$, we can show that $a_{i+1}$ is an arrow on $\C$ by using \ref{G3},
and obtain a dual contradiction. Then $a_{i+1}$ is not on $\C$,
and, furthermore, all arrows $a_i$, $\ldots$, $a_n$ are not on $\C$ by induction.
\end{proof}

\begin{lemma} \label{lemm:main0302}
Keep the notations from Lemma \ref{lemm:recollement CMA}.
If $\tC:=C/C(1-\e_{\nota})C$ is quasi-tilted, then:
\begin{enumerate}[label={\rm(\arabic*)}]
  \item \label{lemm:main0302 (1)} $\findim A \=< 2$,
  \item \label{lemm:main0302 (2)} for any non-forbidden module $M\in\ind(\modcat(A))$
    with $\pdim M <\infty$ and $\idim M <\infty$, we have
    \[ \pdim M + \idim M \=< 3. \]
\end{enumerate}
\end{lemma}

\begin{proof}
(1) If $\findim A \>=3$, then there is an indecomposable module $M\in\modcat(A)$
with finite projective dimension $d=\pdim A\>= 3$.
Thus, $M$ has $d$ non-zero syzygies $\Omega_1(M)$, $\ldots$, $\Omega_d(M)$,
and for each $n\>=d+1$, we have $\Omega_n(M)=0$.
Each non-zero syzygy $\Omega_t(M)$ ($1\=< t\=< d$) has a direct summand $X_t$,
such that $\top(X_1)$, $\ldots$, $\top(X_d)$ are simple modules
and $\M^{-1}(\top(X_1))$, $\ldots$, $\M^{-1}(\top(X_d))$ induce a right maximal forbidden path $F$ on $(\Q,\I)$,
which is of the following form
\[ \xymatrix{v_0 \ar[r]^{\alpha_0} & v_1 \ar[r]^{\alpha_1} & \cdots \ar[r]^{\alpha_{d-2}} & v_{d-1} \ar[r]^{\alpha_{d-1}} & v_d,} \]
where $v_0$ is an ending point of $\sfs$, and $S(v_t) \cong \M^{-1}(\top(X_t))$ holds for all $1\=< t\=< d$.
Therefore, $\alpha_{d-1}$ is not on any forbidden cycle.
It follows that all arrows of $F$ are not on any forbidden cycle by Lemma \ref{lemm:forbcyc}.
Thus, $F$ can be seen as a forbidden path on the bound quiver of $\tC$,
and so, $\gldim (\tC) \>= \ell(F) = d \>=2$ by using \cite[Theorem 5.10]{LGH2024}
(or by using \cite[Theorem 4.7]{FOZ2024}).
We obtain a contradiction since $\tC$ is quasi-tilted.

(2) If there is a non-forbidden module $M\in\modcat(A)$ satisfies $\pdim M + \idim M \>= 4$,
then one of $\pdim M \>=2$ and $\idim M \>=2$ holds.
Let $\sfs$ be a string on $(\Q,\I)$ corresponding to $M \cong \M(\sfs)$.
By (1), we have $\findim A \=< 2$, then by using the fact that $A$, as a gentle algebra, satisfies
\[ \findim A = \sup_{\pdim M < \infty} \pdim M = \sup_{\idim M < \infty} \idim M,  \]
we obtain $\pdim M, \idim M \=< \findim A \=< 2$. It follows
\[\pdim M = \idim M =2.\]
Thus, $\sfs$ satisfies Lemma \ref{lemma:dim=2} \ref{lemma:dim=2 (A)}, \ref{lemma:dim=2 (B)}, and \ref{lemma:dim=2 (C)},
i.e., there is a left maximal forbidden path $\wp=x_1x_2$ ($x_1,x_2\in\Q_1$) with $\t(\wp)=\s(\sfs)$
and there is a right maximal forbidden path $\wp'=y_1y_2$ ($y_1,y_2\in\Q_1$) with $\d(\wp')=\t(\sfs)$.
Thus, one can check that $x_1$, $x_2$, $y_1$, and $y_2$ are not on any forbidden cycle.
Assume $\sfs=s_1\cdots s_n$, where $s_i \in \Q_1^{\pm 1}$, and define
\def\zerostr{\mathsf{0}}
\[
s_i^{\dag} = \begin{cases}
s_i, & \text{if } s_i \in \Q_1 \text{ is not on any forbidden cycle}; \\
s_i, & \text{if } s_i^{-1} \in \Q_1 \text{ is not on any forbidden cycle}; \\
\zerostr, & \text{if } s_i \in \Q_1 \text{ is on a forbidden cycle}; \\
\zerostr, & \text{if } s_i^{-1} \in \Q_1 \text{ is on a forbidden cycle}. \\
\end{cases}
\]
Here, $\zerostr$ ($\notin\Str(A)$) is called a \defines{zero string}, which is used to describe zero module,
i.e., $\M(\zerostr) = 0$ ($\notin\ind(\modcat(A))$).

If $\sfs^{\dag}\ne\zerostr$, then $\sfs^{\dag} := s_1^{\dag}\cdots s_n^{\dag}$
is also a string on the bound quiver of $\tC$,
and $\sfs^{\dag}$ satisfies Lemma \ref{lemma:dim=2} \ref{lemma:dim=2 (A)}, \ref{lemma:dim=2 (B)}, and \ref{lemma:dim=2 (C)}
(indeed, on the bound quiver of $\tC$, $\wp$ and $\wp'$
can be seen as left and right maximal forbidden paths with respect to $\sfs^{\dag}$).
We obtain that $\M(\sfs^{\dag})\in\modcat(\tC)$ is a $\tC$-module with
$\pdim \M(\sfs^{\dag}) + \idim \M(\sfs^{\dag}) = 4$ by Lemma \ref{lemma:dim=2}.
It follows that $\tC$ is not quasi-tilted by Theorem \ref{thm:LMXZ}
(or precisely, by Remark \ref{rmk:LMXZ}). This is a contradiction.

If $\sfs^{\dag}=\zerostr$, then $\ell(\sfs)\>=1$ and there exists one arrow in $\sfs$ such that
this arrow is on some forbidden cycle of $(\Q,\I)$. In this case, $M\cong \M(\sfs)$ is forbidden, a contradiction.

Thus, $\pdim M + \idim M \=< 4$ holds for all $M\in\ind(\modcat(A))$.
\end{proof}

\begin{lemma} \label{lemm:main0303}
Keep the notations from Lemma \ref{lemm:recollement CMA}.
If the statements \ref{lemm:main0302 (1)} and \ref{lemm:main0302 (2)}
given in Lemma \ref{lemm:main0302} hold, then $\tC$ is quasi-titled.
\end{lemma}

\begin{proof}
Notice that in a case of $A$ to be a gentle algebra with $\gldim A=\infty$,
\cite[Corollary 6.8 (i)--(iv)]{B2011} admits
\[ \gldim C = \begin{cases}
\findim A, &\text{if } \findim A \>=2; \\
2, & \text{if } \findim A \=< 2.
\end{cases} \]
Then Lemma \ref{lemm:main0302} \ref{lemm:main0302 (1)} yields $\gldim C = 2$.
Since:
\begin{itemize}
  \item the bound quiver of $\tC$ is obtained by deleting all vertices
$\mathfrak{v}^{-1}(cA)$ ($c\in\Q_1$ is on some forbidden cycle of $(\Q,\I)$)
corresponding to non-projective indecomposable G-projective $A$-modules from $(\Q^{\CMA},\I^{\CMA})$;
  \item and all $\mathfrak{v}^{-1}(cA)$, as vertices on $(\Q^{\CMA},\I^{\CMA})$, are not relational,
\end{itemize}
we have $\gldim \tC \=< \gldim C$. Thus, $\gldim \tC \=< 2$.

Next, assume $(\Q,\I)$ has $r$ forbidden cycles $\C_i=c_{i1}c_{i2}\cdots c_{in_i}$ ($1\=< r\=< r$),
and the bound quiver of $\tC$ is $(\tilde{\Q},\tilde{\I})$.
Let $\tilde{M}$ be an indecomposable $\tC$-module,
and $\tilde{\sfs}$ be a string on the bound quiver of $\tC$ corresponding to $\tilde{M}$.
Notice that $(\tilde{\Q},\tilde{\I})$ is obtained by deleting all vertices
$\mathfrak{v}^{-1}(c_{ij}A)$ (and all arrows starting or ending with $\mathfrak{v}^{-1}(c_{ij}A)$)
from $(\Q^{\CMA}, \I^{\CMA})$, then $\tilde{\sfs}$ can be seen as a string $\sfs$ on $(\tilde{\Q},\tilde{\I})$.
By using $\gldim \tC \=< 2$, we have $\pdim \tilde{M} \leqslant 2$ and $\idim \tilde{M} \leqslant 2$.
Assume $\pdim \tilde{M} = \idim \tilde{M} = 2$.
We obtain that $\tilde{\sfs}$ satisfies Lemma \ref{lemma:dim=2}
\ref{lemma:dim=2 (A)}, \ref{lemma:dim=2 (B)}, and \ref{lemma:dim=2 (C)}
(the length of $\tilde{\sfs}$ may be zero in this proof).
Here, the following two cases need be considered:
\begin{itemize}
  \item[(a)] $\sfs$ has an endpoint on some forbidden cycle of $(\tilde{\Q},\tilde{\I})$;
  \item[(b)] two endpoints are not on any forbidden cycle of $(\tilde{\Q},\tilde{\I})$.
\end{itemize}

Let $M$ be the indecomposable $A$-module corresponding to $\sfs$.
If (a) holds, then one of $\pdim M$ and $\idim M$ is infinite.
If (b) holds, assume $F_{\mathrm{lu}}$, $F_{\mathrm{ru}}$, $F_{\mathrm{ld}}$, $F_{\mathrm{rd}}$
are four forbidden paths on $(\tilde{\Q},\tilde{\I})$ with respect to $\tilde{\sfs}$.
Then, naturally, they can be seen as four forbidden paths on $(\Q,\I)$.
In this case, $F_{\mathrm{lu}}$ as a forbidden path on $(\Q,\I)$ is left maximal.
Otherwise, there is an arrow $a$ with $\t(a)=\s(F_{\mathrm{lu}})$ such that
$aF_{\mathrm{lu}}$ is a forbidden path on $(\Q,\I)$.
Recall that $(\tilde{\Q},\tilde{\I})$ is obtained by deleting all vertices
$\mathfrak{v}^{-1}(c_{ij}A)$ from $(\Q^{\CMA}, \I^{\CMA})$,
we get $\s(a) \in \mathfrak{v}^{-1}(\ind(\Gproj(A)))$,
and so, $\t(a)=\s(F_{\mathrm{lu}})$ is a vertex on a forbidden cycle of $(\Q,\I)$,
a contradiction. One can check that $F_{\mathrm{ru}}$ as a forbidden path on $(\Q,\I)$ is left maximal in a similar way,
and one can check that $F_{\mathrm{ld}}$ and $F_{\mathrm{rd}}$ as forbidden paths on $(\Q,\I)$ are right maximal in a dual way.
Thus, on the bound quiver $(\Q,\I)$, we obtain that:
\begin{itemize}
  \item at least one of $\ell(F_{\mathrm{lu}})$ and $\ell(F_{\mathrm{ru}})$ is $2$ by $\idim M =2$;
  \item and at least one of $\ell(F_{\mathrm{ld}})$ and $\ell(F_{\mathrm{rd}})$ is $2$ by $\pdim M =2$.
\end{itemize}

Therefore, $\pdim M = \idim M = 2$ holds, this contradict with Lemma \ref{lemm:main0302} \ref{lemm:main0302 (2)}.
Then at least one of $\pdim \tilde{M}$ and $\idim \tilde{M}$ is less than of equal to $1$,
i.e., $\tC$ is quasi-tilted.
\end{proof}

\begin{theorem} \label{thm:main}
Let $A=\kk\Q/\I$ be a gentle algebra and let $C$ be its CM--Auslander algebra. Then the quotient $\tC:=C/C(1-\e_{\nota})C$ is quasi-tilted if and only if the following statements hold.
\begin{enumerate}[label=\text{\rm(\arabic*)}]
  \item $\findim A \=< 2$; \label{thm:main (1)}
  \item for any non-forbidden indecomposable $A$-module $M$,
    if $\pdim M<\infty$ and $\idim M <\infty$,
    then $\pdim M + \idim M \=< 3$. \label{thm:main (2)}
\end{enumerate}
\end{theorem}

\begin{proof}
By Lemmas \ref{lemm:main0302} and \ref{lemm:main0303}, we obtain this theorem immediately.
\end{proof}

We provide a remark for Theorem \ref{thm:main} in two trivial cases.

\begin{remark} \rm
If either every module over a gentle algebra $A$ is forbidden, or every forbidden module over
$A$ fails to satisfy at least one of conditions \ref{thm:main (1)} and \ref{thm:main (2)}, then, trivially,
$A$ satisfies both \ref{thm:main (1)} and \ref{thm:main (2)}.
In this case, if $\findim A \=< 2$, then $\tC$ is quasi-tilted.
\end{remark}

\begin{example}\label{examp:quasi-tilted} \rm
Consider the gentle algebra $A=\kk\Q/\I$ given in Example \ref{examp:gentle}.
Example \ref{examp:CMA of gentle} has shown the bound quiver of its CM--Auslander algebras.
\def\lengthone{0.1}
\def\lengthtwo{0.2}
\begin{figure}[htbp]
  \centering
\begin{tikzpicture}[scale=0.5]
\fill[top color=blue!25, bottom color=blue!5]
  ( 7.2,0) to[out=90,in=0] (0,7.2) to[out=180,in=90]
  (-7.2,0) to[out=-90,in=180] (0,-7.2) to[out=0,in=-150] (3.6,-6.2)
  -- (7,-7) -- (6.2,-3.6) to[out=60,in=-90] ( 7.2,0);
\draw[blue] (7,-7) node[below]{$C:=A^{\mathrm{CMA}}$};
\foreach \x in {60,180,300}
\fill[white][rotate around = {\x:(0,0)}][line width=1pt]
  (2,0) circle(1) (6,0) circle(1);
\foreach \x in {60,180,300}
\fill[red!35][rotate around = {\x:(0,0)}][line width=1pt]
  (2,0) circle(0.65) (6,0) circle(0.65);
\foreach \x in {0,120,240}
\draw[green!35][rotate around = {\x:(0,0)}]
  [line width=12pt][->] (1.5,0) -- (8.5,0);
\draw[green][rotate around = {  0:(0,0)}] (9,0) node{$C_1$};
\draw[green][rotate around = {120:(0,0)}] (9,0) node{$C_2$};
\draw[green][rotate around = {240:(0,0)}] (9,0) node{$C_3$};
\foreach \x in {60,180,300}
\draw[red!35][line width=5pt]
  [rotate around = {\x:(0,0)}][->]
  (1.5,0) -- (0.3,0);
\foreach \x in {0,120,240}
\draw[red!35][line width=7pt]
  [rotate around = {\x+60:(0,0)}]
  (5.75,0) -- (2,0);
\draw[red] (0,0) node{$C'$};
\foreach \x in {0,60,120,180,240,300}
\draw[->][rotate around = {10+\x:(0,0)}][line width=1pt] (2,0) arc(0:40:2);
\foreach \x in {0,60,120,180,240,300}
\draw[->][rotate around = {5+\x:(0,0)}][line width=1pt] (6,0) arc(0:50:6);
\draw (2,0) node{$1$} (-1*1,1.73*1) node{$2$} (-1*1,-1.73*1) node{$3$};
\draw (4,0) node{$4$} (-1*2,1.73*2) node{$5$} (-1*2,-1.73*2) node{$6$};
\draw (6,0) node{$7$} (-1*3,1.73*3) node{$8$} (-1*3,-1.73*3) node{$9$};
\foreach \x in {0,120,240}
\draw[->][rotate around = {\x:(0,0)}][line width=1pt] (3.7,0) -- (2.3,0);
\foreach \x in {0,120,240}
\draw[->][rotate around = {\x:(0,0)}][line width=1pt] (5.7,0) -- (4.3,0);
\draw[rotate around = { 60-30:(0,0)}] (2.4+\lengthone,0) node{$a_{12}^{-}$};
\draw[rotate around = {180-30:(0,0)}] (2.4+\lengthone,0) node{$a_{23}^{-}$};
\draw[rotate around = {300-30:(0,0)}] (2.4+\lengthone,0) node{$a_{31}^{-}$};
\draw[rotate around = { 60-30:(0,0)}] (6.4+\lengthtwo,0) node{$a_{78}^{-}$};
\draw[rotate around = {180-30:(0,0)}] (6.4+\lengthtwo,0) node{$a_{89}^{-}$};
\draw[rotate around = {300-30:(0,0)}] (6.4+\lengthtwo,0) node{$a_{97}^{-}$};
\draw[rotate around = { 60+30:(0,0)}] (2.4+\lengthone,0) node{$a_{12}^{+}$};
\draw[rotate around = {180+30:(0,0)}] (2.4+\lengthone,0) node{$a_{23}^{+}$};
\draw[rotate around = {300+30:(0,0)}] (2.4+\lengthone,0) node{$a_{31}^{+}$};
\draw[rotate around = { 60+30:(0,0)}] (6.4+\lengthtwo,0) node{$a_{78}^{+}$};
\draw[rotate around = {180+30:(0,0)}] (6.4+\lengthtwo,0) node{$a_{89}^{+}$};
\draw[rotate around = {300+30:(0,0)}] (6.4+\lengthtwo,0) node{$a_{97}^{+}$};
\draw[rotate around = {  0:(0,0)}] (3,0.5) node{$a_{41}$};
\draw[rotate around = {120:(0,0)}] (3,0.5) node{$a_{52}$};
\draw[rotate around = {240:(0,0)}] (3,0.5) node{$a_{63}$};
\draw[rotate around = {  0:(0,0)}] (5,0.5) node{$a_{74}$};
\draw[rotate around = {120:(0,0)}] (5,0.5) node{$a_{85}$};
\draw[rotate around = {240:(0,0)}] (5,0.5) node{$a_{96}$};
\foreach \x in {0,120,240}
\draw[line width=1.4pt][dashed][red]
[rotate around = {\x:(0,0)}] (1.85,-0.5) arc(-90:-270:0.5);
\foreach \x in {0,120,240}
\draw[line width=1.4pt][dashed][red]
[rotate around = {\x:(0,0)}] (5.95,-1) arc(-90:90:1);
\foreach \x in {0,120,240}
\draw[line width=1.4pt][dashed][violet]
[rotate around = {\x:(0,0)}] (3,0) arc(-180:0:1);
\draw[rotate around = { 60:(0,0)}] (2,0) node{\footnotesize$G_{12}$} (-1*1,1.73*1);
\draw[rotate around = {180:(0,0)}] (2,0) node{\footnotesize$G_{23}$} (-1*1,1.73*1);
\draw[rotate around = {300:(0,0)}] (2,0) node{\footnotesize$G_{31}$} (-1*1,1.73*1);
\draw[rotate around = { 60:(0,0)}] (6,0) node{\footnotesize$G_{78}$} (-1*1,1.73*1);
\draw[rotate around = {180:(0,0)}] (6,0) node{\footnotesize$G_{89}$} (-1*1,1.73*1);
\draw[rotate around = {300:(0,0)}] (6,0) node{\footnotesize$G_{97}$} (-1*1,1.73*1);
\end{tikzpicture}
  \caption{\texttt{Recollements of $\modcat(\tilde{C})$, $\modcat(C)$ and $\modcat(A)$}}
  \label{fig:Quasitilted}
\end{figure}
Then we have a recollement
\begin{align*}
\calR(C,1-\e_{\nota}):= \
\xymatrix@C=2cm{
  \modcat(\tC) \ar[r]^{\mathrm{embedding}}
& \modcat(C) \ar[r]^{(-)(1-\e_{\nota})}
  \ar@/_1.5pc/[l]_{-\otimes_{C} \tC}
  \ar@/^1.5pc/[l]^{\Hom_{C}(\tC, -)}
& \modcat(C'),
  \ar@/_1.5pc/[l]_{-\otimes_{C'} (1-\e_{\nota})C}
  \ar@/^1.5pc/[l]^{\Hom_{C'}(C(1-\e_{\nota}),-)}
}
\end{align*}
where $\e_{\nota} = e_{G_{12}}+e_{G_{23}}+e_{G_{31}}+e_{G_{78}}+e_{G_{89}}+e_{G_{97}}$.
In \Pic \ref{fig:Quasitilted}, the shadow part marked by ``\shadow{red!35}'' is $C'$,
the shadow part marked by ``\shadow{blue!35}'' is $C$,
and the shadow part marked by ``\shadow{green!35}'' is $\tC$ ($= C_1\times C_2 \times C_3$)
$\cong (\kk(\xymatrix{1 \ar[r]^{a} & 2 \ar[r]^b & 3})/\langle ab\rangle)^{\times 3}$.
In this example, it is clear that $\tC$ is quasi-tilted.
For the algebra $A$, one can check that it is representation-finite
and its Auslander--Reiten quiver is shown in \Pic \ref{fig:ARquiver}.
\begin{figure}[htbp]
  \centering
  \includegraphics[width=9cm]{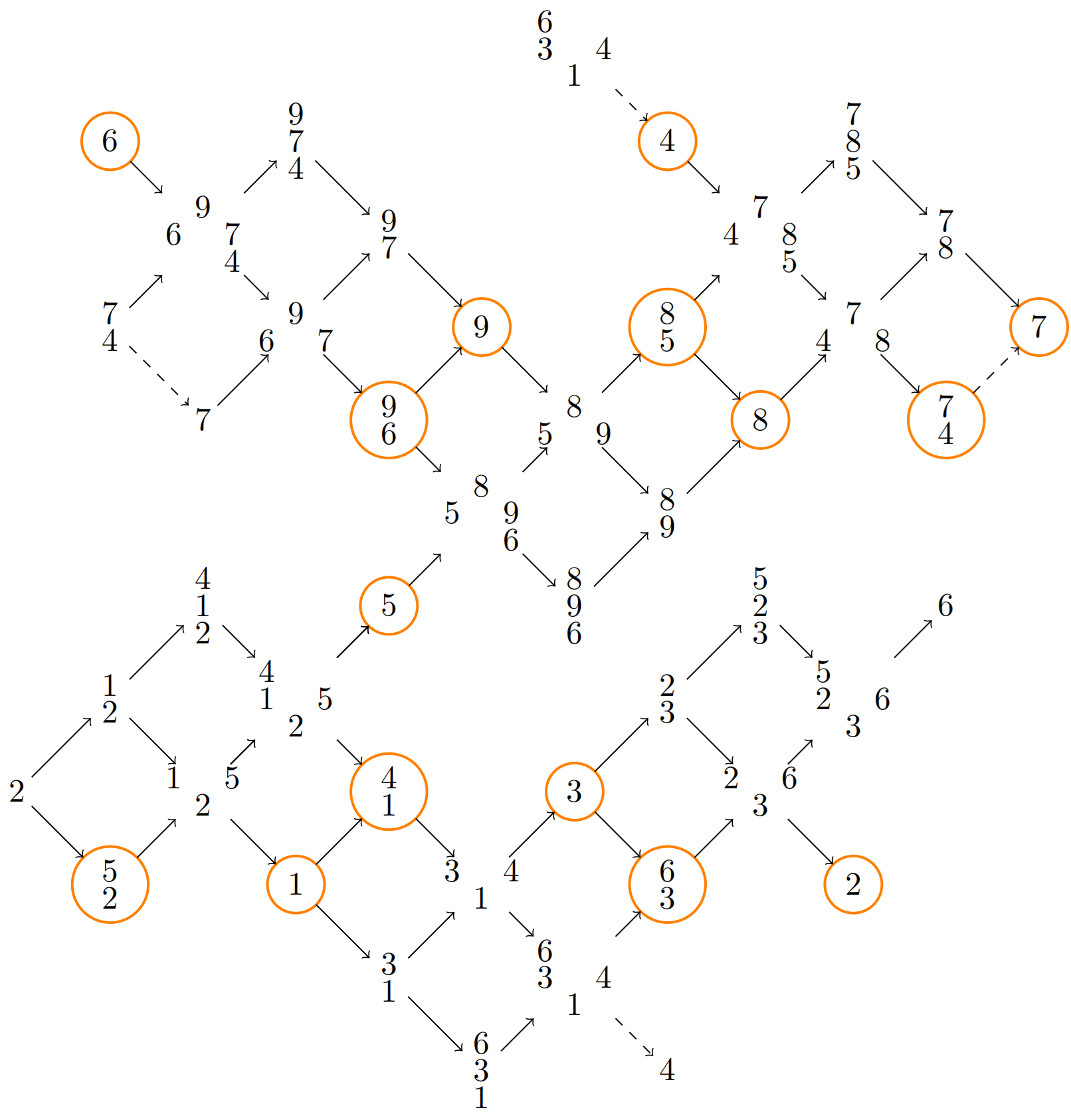}\\
  \caption{\texttt{The Auslander--Reiten quiver of the gentle algebra $A$ given in Example \ref{examp:gentle}
  (the modules marked by ``${\color{orange}\pmb{\bigcirc}}$'' are non-forbidden modules)}}
  \label{fig:ARquiver}
\end{figure}
The algebra $A$ has 15 indecomposable non-forbidden modules:
\begin{align}
 & \sm{1}, && \sm{4}, && \sm{7}, && \sm{4\\1}, && \sm{7\\4}; & \label{in examp quasi-tilted 1} \\
 & \sm{2}, && \sm{5}, && \sm{8}, && \sm{5\\2}, && \sm{8\\5}; & \label{in examp quasi-tilted 2} \\
 & \sm{3}, && \sm{6}, && \sm{9}, && \sm{6\\3}, && \sm{9\\6}; & \label{in examp quasi-tilted 3}
\end{align}
where (\ref{in examp quasi-tilted 1}), (\ref{in examp quasi-tilted 2}), and (\ref{in examp quasi-tilted 3})
are provided by the subquiver $1 \mathop{\longleftarrow}\limits^{a_{41}} 4 \mathop{\longleftarrow}\limits^{a_{74}} 7$,
$2 \mathop{\longleftarrow}\limits^{a_{52}} 5 \mathop{\longleftarrow}\limits^{a_{85}} 8$,
and $3 \mathop{\longleftarrow}\limits^{a_{63}} 6 \mathop{\longleftarrow}\limits^{a_{96}} 9$, respectively.
One can check that the projective resolutions of $\sm{1}$, $\sm{4\\1}$ and $\sm{7}$ are infinite,
then, by the symmetry of the bound quiver of $A$, we obtain
$\pdim\sm{2}$, $\pdim \sm{5\\2}$, $\pdim\sm{8}$, $\pdim\sm{3}$, $\pdim\sm{6\\3}$ and $\pdim\sm{9}$ are infinite.
Dually, one can check that $\idim\sm{7\\4} = \infty$,
then, by the symmetry of the bound quiver of $A$, we obtain $\idim\sm{8\\5}$ and $\idim\sm{9\\6}$ are infinite.
Moreover, we have the projective resolution of ${\sm{4}}$ is
$$	0 \longrightarrow P(1)_A={\sm{1\\2} } \longrightarrow P(4)_A={\sm{4 \\1\\2}} \longrightarrow {\sm{4}} \to 0, $$
and the injective resolution of ${\sm{4}}$ is
$$	0 \longrightarrow {\sm{4}} \longrightarrow E(4)={\sm{9 \\7\\4} } \longrightarrow E(7) = {\sm{9\\7}} \to 0, $$
then $\pdim\sm{4} + \idim\sm{4} = 2~ (\=< 3)$.
We can obtain $\pdim\sm{5} + \idim\sm{5} = 2$ and $\pdim\sm{6} + \idim\sm{6} = 2$
by the symmetry of the bound quiver of $A$.
Therefore, $A$ satisfies the conditions \ref{thm:main (1)} and \ref{thm:main (2)} given in Theorem \ref{thm:main}.
\end{example}

\section{Derived representation types}

The derived category $\Dcat^b(A)$ of a gentle algebra $A$ (or some graded gentle algebra)
have been studied by \cite[etc]{AG2008,ALP2016,KY2018}.
In \cite{BM2003,ALP2016}, the authors described objects and morphisms in $\Dcat^b(A)$ by using homotopy strings and homotopy bands,
this result later became the basis for studying the derived representations of algebras,
see for example, \cite[etc]{ZH2016, Zhang2019Indecomposables}.

\subsection{Homotopy strings and homotopy bands}

Let $(\Q,\I)$ be a gentle pair, $\Q_l$ be the set of all non-zero paths in $(\Q,\I)$ of length $l$.
and $\Q_l^{-1}$ be the set of all formal inverses of non-zero paths in $(\Q,\I)$ of length $l$.
A \defines{homotopy string} on $(\Q,\I)$ is a sequence $\sfh=(\wp_i)_{i=1}^{r}$ in $\Q_{l}^{\pm 1} := \Q_l\cup\Q_l^{-1}$,
where $\wp_i=a_{i1}\cdots a_{il_i}$ is either a non-zero path or a formal inverse such that
for any non-zero path (resp., a formal inverse) $\wp_i$,
one of the following statements holds ($i\=< r-1$):
\begin{itemize}
  \item $\wp_{i+1}$ is a formal inverse of a path satisfying $\wp_{i+1}^{-1}\ne\wp_i$
    (resp., $a_{i+1\ 1}^{-1}~a_{il_i}^{-1}\in \I$);
  \item $\wp_{i+1}$ is a path satisfying $a_{il_i}a_{i+1~1}\in \I$
    (resp., $\wp_{i+1}^{-1}\ne\wp_i$).
\end{itemize}
A \defines{homotopy band} $\sfh=(\wp_i)_{i=1}^{r}$ is a homotopy string such that:
\begin{itemize}
  \item $\t(\sfh)=\s(\sfh)$;
  \item if $\wp_1$ and $\wp_r$ both are paths (resp., formal inverse),
    then $a_{rl_r}a_{1~1} ~\in\I$ (resp., $a_{1~1}^{-1}~a_{rl_r}^{-1}\in\I$);
  \item the number of paths in $\sfh$ equals to that of formal inverses in $\sfh$;
  \item $\sfh$ is not a non-trivial power of any homotopy string.
\end{itemize}

We call two homotopy strings $\sfh_1$ and $\sfh_2$ are equivalent if $\sfh_1=\sfh_2$ or $\sfh_1=\sfh_2^{-1}$ holds;
and we call two homotopy bands $\sfh_1$ and $\sfh_2$ are equivalent if $\sfh_1=\sfh_2[t]$ or $\sfh_1=\sfh_2^{-1}[t]$ holds,
where for each $\sfh=\wp_1\cdots \wp_l$, $\sfh[t] := \wp_{1+t}\cdots \wp_l\wp_1\cdots \wp_t$ ($0\=< t< l$).
Homotopy strings and homotopy bands describe all indecomposable objects in $\Dcat^b(A)$ of a gentle algebra $A$.

\begin{theorem}[{Bekkert--Merklen \cite[Theorem 3]{BM2003}}] \label{thm:BM2003}
Let $A=\kk\Q/\I$ be a gentle algebra and $\hStr(A)$ {\rm(}resp., $\hBand(A)${\rm)}
be the set of all equivalence classes of homotopy strings {\rm(}resp., homotopy bands{\rm)}.
Then there is a bijection
\[\X: (\hStr(A)\times\NN)\cap(\hBand(A)\times\mathscr{J}\times\NN) \to \ind(\Dcat^b(A))\]
from the disjoint union $ (\hStr(A)\times\NN)\cap(\hBand(A)\times\mathscr{J}\times\NN)$ to
the set of all isoclasses of indecomposable complexes in $\Dcat^b(A)$,
where $\mathscr{J}$ is the set of all Jordan blocks with non-zero eigenvalues.
\end{theorem}

\begin{notation} \rm
For any homotopy string or band with a Jordan block $\pmb{J}$, denoted as $\sfh$,
we define $\X((\sfh,n)):=\X(([\sfh],n))$ and $\X((\sfh,\pmb{J},n)):=\X(([\sfh],\pmb{J},n))$.
\end{notation}

This theorem provides two important facts as follows:

\begin{itemize}
  \item $\ind(\modcat(A)) = \image(\X|_{\hStr(A)\times\NN}) \cup \image(\X|_{\hBand(A)\times\NN\times\mathscr{J}})$;
  \item $\image(\X|_{\hStr(A)\times\NN}) \cap \image(\X|_{\hBand(A)\times\NN\times\mathscr{J}}) = \varnothing$.
\end{itemize}

In \cite{BM2003}, each indecomposable complex lying in $\image(\X|_{\hStr(A)\times\NN})$ is called a \defines{string complex},
and each indecomposable module lying in $\image(\X|_{\hBand(A)\times\NN\times\mathscr{J}})$ is called a \defines{band complex}.

\subsection{Cohomological widths}

Let $\comp{X}=\xymatrix{\cdots \ar[r] & X^{-1} \ar[r]^{d^{-1}} & X^0 \ar[r]^{d^0} & X^1 \ar[r]^{d^1} & \cdots}$
be an indecomposable object in the derived category $\Dcat^b(A)$ of an algebra.
Its \defines{cohomological width} is defined as
\[ \hw\comp{X} := \sup \{ j-i+1 \mid \coH^i(\comp{X})\ne 0, \coH^j(\comp{X})\ne 0, i \=< j \}, \]
and, furthermore, the \defines{global cohomological width} of $A$ is defined as
\[ \glhw(A) : = \sup\{\hw\comp{X} \mid \comp{X}\in\ind(\Dcat^b(A))\}. \]

Global cohomological width provides a description of an algebra to be quasi-tilted,
this result is based on \cite[Proposition 3.3]{H2008}, and it is first showed in \cite{Z2017globalcohomo}.

\begin{lemma}[{\cite[Proposition 3.3]{H2008}, \cite[Proposition 4.5]{Z2017globalcohomo}}] \label{lemm:HZ}
Let $A$ be an algebra. Then $\glhw A = 2$ if and only if
$A$ is a quasi-tilted algebra which is not hereditary.
\end{lemma}

\begin{theorem} \label{thm:main2}
Let $A=\kk\Q/\I$ be a gentle algebra. The following statements are equivalent.
\begin{itemize}
  \item[\rm(1)]
    If $\sfh$ is a homotopy string {\rm(}or a homotopy band{\rm)}
    whose arrows are not on any forbidden cycle, then
	$\hw\X(([\sfh],n)) \=< 2$ {\rm(}or $\hw\X(([\sfh],\pmb{J},n))\=<2${\rm)} holds for any $n\in\NN$;
  \item[\rm(2)]
    For any non-forbidden module $M \in \ind(\modcat(A))$ with $\pdim M \=< +\infty$ and $\idim M \=< +\infty$, we have
    $\findim A \=< 2$ and $\pdim M + \idim M \=< 3$.
\end{itemize}
\end{theorem}

\begin{proof}
$(1) \Rightarrow (2)$:
First, if $\findim A> 2$, then there exists a forbidden path $\sfh=a_1a_2a_3 = $
$1 \To{a_1} 2\To {a_2} 3 \To{a_3} 4$
on $(\Q,\I)$ such that all arrows of it are not on any forbidden cycle.
Then $\sfh$, as a homotopy string on $(\Q,\I)$, corresponds to an indecomposable complex
\[
\X(\sfh) :=
  \cdots \To{} 0
  \To{} P(3)
  \To{f_2}
P(2)
	\To{f_1}
P(1)
  \To{} 0
  \To{} \cdots
\]
in $\Dcat^b(A)$ by using Theorem \ref{thm:BM2003}.
In this case, we have the following two facts:
\begin{enumerate}[label=\textrm{(\alph*)}]
  \item \label{thm:main2(1)}
    $f_2: P(3)=e_3A \to P(2)=e_2A$, $e_3x\mapsto a_2e_3x=e_2a_2x$ ($\forall x\in A$) is not monomorphic
    by the fact that $a_3 \in \kernel(f_2)$ admits $\kernel(f_2) \ne 0$.
  \item \label{thm:main2(2)}
    $f_1: P(2)=e_2A \to P(1)=e_1A$, $e_2x\mapsto a_1e_2x=e_1a_1x$ ($\forall x\in A$) is not epimorphic.
    Otherwise, $f_1$ is a projective precover of $P(1)$, which follows that $P(1)$ is a direct summand of $P(2)$,
    i.e., $P(2)$ is decomposable, a contradiction.
\end{enumerate}
Therefore, $\hw\X(([\sfh],n)) = 3$ holds for all $n\in\NN$. We obtain a contradiction, i.e., $\findim A \=< 2$.

Second, for any non-forbidden module $M\in\ind(\modcat(A))$ with $\pdim M <+\infty$ and $\idim M <+\infty$,
we assume $\pdim M +\idim M >3$. Then $\pdim M +\idim M =4$, otherwise, one of $\pdim M\>=3$ and $\idim \>= 3$ holds,
which contradicts with $\findim A \=< 2$. Furthermore, we have $\pdim M = \idim M = 2$ by using $\pdim M\>=3$.
In this case, there is a string $\sfs$ corresponding to $M$ such that
Lemma \ref{lemma:dim=2} \ref{lemma:dim=2 (A)}, \ref{lemma:dim=2 (B)}, and \ref{lemma:dim=2 (C)} hold.
Then there is a reduced walk which is of the form $\alpha\beta\sfs\gamma\delta$
($\alpha,~\beta,~\gamma,~\delta\in\Q_1$) such that
\begin{itemize}
  \item $\beta\sfs \gamma$ is a string;
  \item $\alpha\beta$ and $\gamma\delta$ respectively are left and right maximal forbidden paths.
\end{itemize}
Since all arrows of $\sfs$ are not on any forbidden cycle of $(\Q,\I)$,
then all arrows of $\alpha\beta\sfs\gamma\delta$  are not on any forbidden cycle of $(\Q,\I)$
(by Lemma \ref{lemm:forbcyc}). There are four cases as follows.
\begin{figure}[H]
  \centering
  \includegraphics[width=0.98\textwidth]{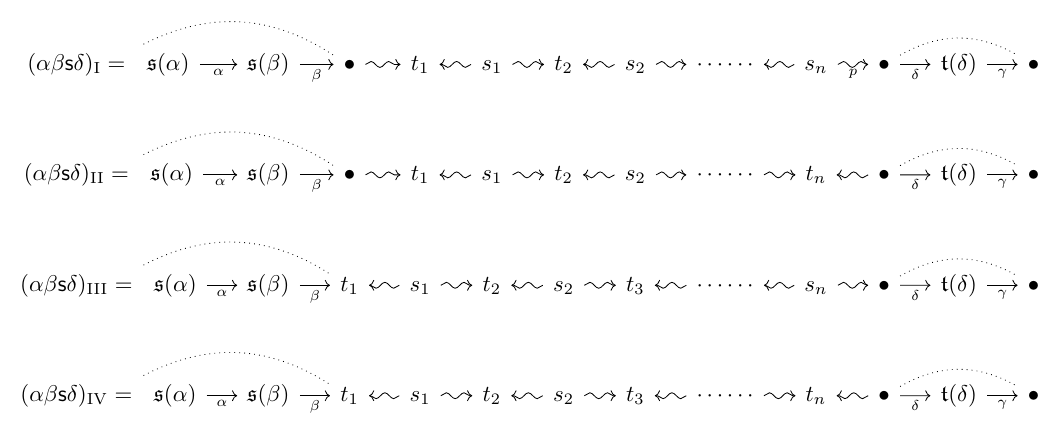}
\end{figure}
We only consider the case of $\alpha\beta\sfs\delta=(\alpha\beta\sfs\delta)_{\rmI}$,
the other cases are similar. Notice that we have
\begin{align*}
  \X(([\alpha\beta\sfs\gamma\delta],0))
& \cong \cdots \To{} 0 \To{} P(t_1)\oplus P(t_n) \oplus P(\t(\delta))
  \To{f} P(\s(\beta))\oplus P(s_1) \oplus \cdots \oplus P(s_n) \\
&  \To{(\varphi,0,\cdots,0)} P(\s(\alpha)) \To{} 0 \To{} \cdots,
\end{align*}
where $f$ is of the form
\begin{center}
  $\left(\begin{matrix}
  * & 0 & \cdots & 0 \\
  * & * & \dots & 0 \\
  \vdots & \vdots && \vdots \\
  * & * & \dots & \phi
\end{matrix}\right)_{(n+1)\times (n+1)}$,
\end{center}
and $\varphi$ and $\phi$ are homomorphisms between two indecomposable projective modules induced by
$\alpha$ and $p$, respectively.
In this case, we can show the following two facts in a way similar to the proof of \ref{thm:main2(1)} and \ref{thm:main2(2)}:
\begin{enumerate}[label=\textrm{(\alph*$'$)}]
  \item $\varphi$ is not epimorphic.
  \item $\phi$ is not monomorphic (then $f$ is not monomorphic).
\end{enumerate}
Then we obtain a indecomposable complex in $\Dcat^b(A)$ whose cohomological width is $3$, a contradiction.

$(2)\Rightarrow (1)$:
By Theorem \ref{thm:main}, (2) admits $\tC := A^{\CMA}/A^{\CMA}(1-\e_{\nota})A^{\CMA}$ is quasi-tilted.
Then we have $\glhw(\tC)\=< 2$ by Lemma \ref{lemm:HZ}.
It follows that $\hw\comp{X}\=<2$ holds for any indecomposable complex in $\Dcat^b(A)$,
and then, for any homotopy string $\sfh$, $\hw \X(([\sfh], n)) \=<2$ holds for all $n\in\NN$.
\end{proof}


\subsection{Krull--Gabriel dimensions}

Let $\Dcat$ be an essentially small triangulated category.
Recall that the \defines{Abelianization} of $\Dcat$,
denoted by $\Ab(\Dcat)$, is the category of all additive contravariant functors $F : \Dcat^{\op} \to \Ab$
from $\Dcat^{\op}$ to the Abelian group category $\Ab$ that are finitely presented,
i.e., for each such $F$, there exists an exact sequence
\[ \Hom_{\Dcat}(-, X) \to \Hom_{\Dcat}(-, Y) \to F \to 0 \]
for some objects $X, Y \in \Dcat$, see \cite[Section 1]{Geigle1985}.

The \defines{Krull--Gabriel filtration} of the Abelianization $\Ab(\Dcat)$ of $\Dcat$ is a sequence of Serre subcategories $\mathcal{A}_n$ of $\Ab(\Dcat)$ indexed by $n \in \mathbb{Z}_{\>= -1}$. Here,
\begin{enumerate}[label=\textrm{(\arabic*)}]
  \item a Serre subcategory of $\Ab(\Dcat)$ is a subcategory such that
    it is closed under subobjects, quotients, and extensions;
  \item $\mathcal{A}_{-1} := \{0\}$;
  \item and, for any $n \>= 0$, $\mathcal{A}_n$ is the full subcategory generated by all objects $X$ in $\Ab(\Dcat)$
    which are of finite length in the image of the functor $q_n : \Ab(\Dcat) \to \Ab(\Dcat)/\mathcal{A}_{n-1}$,
    where $\Ab(\Dcat)/\mathcal{A}_{n-1}$ is a quotient category of $\Ab(\Dcat)$.
\end{enumerate}

\begin{definition} \rm
The \defines{Krull--Gabriel dimension} $\KGdim \Dcat$
of an essentially small triangulated category $\Dcat$ is the smallest integer $t$ such that $\mathcal{A}_t$,
a Serre subcategories in the Krull--Gabriel filtration $\{\mathcal{A}_n\}_{n=1}^{\infty}$ of $\Dcat$, coincides with $\Ab(\Dcat)$.
\end{definition}

A finite-dimensional algebra $A$ is said to be \defines{derived discrete} if for any integer vector
$(d_i)_{i\in\ZZ}$ with finitely many nonzero entries,
the number of indecomposable objects $\comp{X}$ in $\Dcat^b(A)$
satisfying $\dim\coH^i(\comp{X}) = d_i$ for all $i\in\ZZ$ is finite (up to isomorphism and shift), see \cite[Section 1.1]{V2001}.
Otherwise, Han and Zhang proved that $A$ must be \defines{strongly derived unbounded}, i.e.,
there exists a family of integer vectors $\{\pmb{d}_j := (d_{ji})_{i\in\ZZ}\}_{j=1}^{\infty}$
satisfying following two conditions:
\begin{itemize}
  \item any $\pmb{d}_j$ has only a finite number of non-zero components
    (in particular, we define $D_j=\max\limits_{i\in\ZZ} d_{ji}$);
  \item $D_j \max\limits_{i,i'\in\ZZ\atop d_{ji}\ne 0\ne d_{ji'}} (i-i'+1)
    < D_l\max\limits_{i,i'\in\ZZ\atop d_{li}\ne 0\ne d_{li'}} (i-i'+1)$
    holds for all $j<l\in\NN^+$,
\end{itemize}
such that the number of indecomposable objects $\comp{X}_j$ in $\Dcat^b(A)$
satisfying $\dim\coH^i(\comp{X}_j) = d_{ji}$ for all $i\in\ZZ$ is infinite (up to isomorphism and shift),
see \cite[Theorem 2]{ZH2016}.
Vossieck show that a derived discrete finite-dimensional algebra is either a piecewise hereditary algebra
or a gentle one-cycle algebra which is not piecewise hereditary in \cite[Theorem in page 171]{V2001}.
Thus, we immediately obtain the following theorem, which provides a classification of
finite-dimensional algebras up to derived representation types.

\begin{theorem}[{\cite[Theorem in page 171]{V2001},\cite[Theorem 2]{ZH2016}}] \label{thm:VHZ}
A finite-dimensional algebra must satisfy one of the following properties.
\begin{enumerate}[label=\text{\rm(\arabic*)}]
  \item It is both a derived discrete algebra and a piecewise hereditary algebra
    \footnote{We say a finite-dimensional algebra is \defines{piecewise hereditary} if it is derived equivalent to a hereditary algebra, which are introduced by Happel, Rickard, and Schofield in \cite{HRS1988}.}.
  \item It is a derived discrete gentle one-cycle algebra but not piecewise hereditary.
  \item It is a strong derived unbounded algebra.
\end{enumerate}
\end{theorem}

The following result was shown by Bobi{\'n}ski and Krause, which provides a classification of
finite-dimensional algebras by using Krull--Gabriel dimensions.

\begin{theorem}[{\cite[Main Theorem]{BK2015}}] \label{thm:BK}
Let $\mathit{\Lambda}$ be a finite-dimensional algebra.
\begin{enumerate}[label=\text{\rm(\arabic*)}]
  \item If $\mathit{\Lambda}$ is derived discrete and piecewise hereditary,
    then $\KGdim\Ab(\Dcat^b(A))=0$. \label{thm:BK(1)}
  \item If $\mathit{\Lambda}$ is derived discrete and not piecewise hereditary,
    then $\gldim A < \infty$ implies $\KGdim\Ab(\Dcat^b(A))=2$,
    and $\gldim A = \infty$ implies $\KGdim\Ab(\Dcat^b(A))=1$. \label{thm:BK(2)}
  \item If $\mathit{\Lambda}$ is not derived discrete, then $\KGdim\Ab(\Dcat^b(A))\>= 2$. \label{thm:BK(3)}
\end{enumerate}
\end{theorem}

Now we provide the last main result of our paper.

\begin{corollary} \label{coro:main3}
Let $A=\kk\Q/\I$ be a gentle one-cycle algebra and $C$ be its CM--Auslander algebra.
\begin{enumerate}[label=\text{\rm(\arabic*)}]
  \item $\KGdim\Ab(\Dcat^b(A)) \=< 2$ if and only if $\KGdim\Ab(\Dcat^b(C))\=<2$.
  \item $\KGdim\Ab(\Dcat^b(A))=\KGdim\Ab(\Dcat^b(C))$ if and only if $A\cong C$.
  \item $\KGdim\Ab(\Dcat^b(C/C(1-\e_{\nota})C)) = 0$.
\end{enumerate}
\end{corollary}

\begin{proof}
(1) Assume $\KGdim\Ab(\Dcat^b(A)) \=< 2$. We have the following cases:
\begin{enumerate}[label=\text{\rm(\alph*)}]
  \item $A$ is strongly derived unbounded. \label{coro:main3-pf-(a)}
  \item $A$ is derived discrete. \label{coro:main3-pf-(b)}
\end{enumerate}

In the case \ref{coro:main3-pf-(a)}, $(\Q,\I)$ has a homotopy band, it follows that $(\Q,\I)$ has no oriented cycle
since $(\Q,\I)$ has only one cycle.
Then, naturally, $(\Q,\I)$ has no forbidden cycle. In this case, $\gldim A <\infty$, it implies $A=C$.
We obtain $\KGdim\Ab(\Dcat^b(A))=\KGdim\Ab(\Dcat^b(C))$ as required.

In the case \ref{coro:main3-pf-(b)}, $A$ is derived discrete.
Then, by Theorem \ref{thm:VHZ}, $A$ is either a piecewise hereditary algebra or a gentle one-cycle algebra which is not piecewise hereditary.
If it is piecewise hereditary, then $\KGdim(\Ab(\Dcat^b(A)))=0$ by Theorem \ref{thm:BK} \ref{thm:BK(1)}.
In this case, $\gldim A < \infty$, and so, $A=C$. We obtain $\KGdim(\Ab(\Dcat^b(A)))=\KGdim(\Ab(\Dcat^b(C)))=0$ as required.
Otherwise, $A$ is a gentle one-cycle algebra which is not piecewise hereditary.
We have two subcases as follows.
\begin{enumerate}[label=\text{\rm(b.\arabic*)}]
  \item If $\gldim A < \infty$, then $A=C$, i.e., $\KGdim(\Ab(\Dcat^b(A)))=\KGdim(\Ab(\Dcat^b(C)))=0$.
  \item If $\gldim A = \infty$, since $A$ is derived discrete and not piecewise hereditary in this case,
    we obtain $\KGdim(\Ab(\Dcat^b(A)))=1$ by Theorem \ref{thm:BK} \ref{thm:BK(2)}.
    On the other hand, by \cite[Corollary 6.8 (i)--(iv)]{B2011}, we have $\gldim C=\findim A<\infty$,
    and by \cite[Corollary 4.8]{LZhang2023CM-Auslander},
    we have known that $A$ is derived discrete if and only if so is its CM--Auslander algebra $C$.
    Thus, $\KGdim(\Ab(\Dcat^b(C)))=2$.
\end{enumerate}

Conversely, we assume $\KGdim\Ab(\Dcat^b(A)) \=< 2$. We have the following cases:
\begin{enumerate}[label=\text{\rm(\alph*$'$)}]
  \item $A\cong C$. In this case we have $\gldim A < \infty$, this is a trivial case.
  \item $A\not\cong C$. In this case we have $\gldim A = \infty$,
    and then the bound quiver $(\Q^{\CMA},\I^{\CMA})$ of $C$ has only one cycle which is oriented.
    Clearly, $(\Q^{\CMA},\I^{\CMA})$ has no homotopy band. Then $C$ is derived discrete and not piecewise hereditary.
    By Theorem \ref{thm:BK} \ref{thm:BK(2)}, $\KGdim\Ab(\Dcat^b(C))=1<2$ as required.
\end{enumerate}

(2) This statement is a direct consequence of (1).

(3) Notice that the bound quiver of $C/C(1-\e_{\nota})C$ is a gentle tree,
then it is derived discrete and piecewise hereditary.
It follows that $\KGdim\Ab(\Dcat^b(C/C(1-\e_{\nota})C)) = 0$ by Theorem \ref{thm:BK} \ref{thm:BK(1)}.
\end{proof}




\def\cprime{$'$}



\end{document}